\documentclass[12pt]{amsart}
\usepackage{floatflt}
\usepackage{latexsym}
\usepackage{amssymb}
\usepackage[cp850]{inputenc}
\usepackage{epsfig}
\usepackage{psfrag}
\usepackage{amsthm}
\usepackage{amscd}
\usepackage{amsmath}
\usepackage{amsfonts}
\usepackage{graphics}
\usepackage[all]{xy}

\newtheorem{sat}{Theorem}[section]		\newtheorem{lem}[sat]{Lemma}
\newtheorem{kor}[sat]{Corollary}			\newtheorem{prop}[sat]{Proposition}
\newtheorem{bei}{Example}				
\newtheorem*{claim}{Claim}

\newtheorem*{defi*}{Definition}			\newtheorem*{bei*}{Example}
\newtheorem*{sat*}{Theorem}				\newtheorem*{kor*}{Corollary}
\newtheorem*{rmk*}{Remark}					

\let\ssection=\section
\renewcommand{\section}{\setcounter{equation}{0}\ssection}

\newtheorem*{namedtheorem}{\theoremname}
\newcommand{\theoremname}{testing}
\newenvironment{named}[1]{\renewcommand{\theoremname}{#1}\begin{namedtheorem}}{\end{namedtheorem}}

\theoremstyle{remark}
\newtheorem*{bem}{Remark}

\newcommand{\BC}{\mathbb C}			\newcommand{\BH}{\mathbb H}
\newcommand{\BR}{\mathbb R}			
\newcommand{\BN}{\mathbb N}			
\newcommand{\BS}{\mathbb S}			\newcommand{\BZ}{\mathbb Z}

\newcommand{\calD}{\mathcal D}
\newcommand{\CE}{\mathcal E}		
		\newcommand{\CH}{\mathcal H}
		
		\newcommand{\CL}{\mathcal L}
\newcommand{\CM}{\mathcal M}		\newcommand{\CN}{\mathcal N}
		\newcommand{\CP}{\mathcal P}
		
		\newcommand{\CT}{\mathcal T}

\newcommand{\calcium}{\mathcal C \CH}	

\newcommand{\actson}{\curvearrowright}
\newcommand{\D}{\partial}
\newcommand{\DD}{{\nabla}}

\DeclareMathOperator{\PSL}{PSL}		%	Spezielle lineare Gruppe
\DeclareMathOperator{\Id}{Id}		%	Identit\"at
\DeclareMathOperator{\Isom}{Isom}	%	Isometrien einer Mf
\DeclareMathOperator{\Hom}{Hom}		%	Homomorphismen
\DeclareMathOperator{\inj}{inj}

\DeclareMathOperator{\Ker}{Ker}
\DeclareMathOperator{\cusp}{cusp}

\DeclareMathOperator{\length}{length}

\newcommand{\Hyp}{{\mathbb H}}

\DeclareMathOperator{\PML}{\CP \CM \CL }	
\newcommand{\comment}[1]{}

\begin{document}

\title[]{Algebraic and geometric convergence of discrete representations into $\PSL_2\BC$}
\author{Ian Biringer \& Juan Souto}
\thanks{The first author has been partially supported by NSF postdoctoral fellowship DMS-0902991. The second author has been partially supported by the NSF grant DMS-0706878 and the Alfred P. Sloan Foundation.}
\begin{abstract}
Anderson and Canary have shown that if the algebraic limit of a sequence of discrete, faithful representations of a finitely generated group into $\PSL_2\BC$ does not contain parabolics, then it is also the sequence's geometric limit. We construct examples that demonstrate the failure of this theorem for certain sequences of unfaithful representations, and offer a suitable replacement. 
\end{abstract}
\maketitle

\section{Introduction}
When $\Gamma $ is a finitely generated group, let $\calD (\Gamma) $ be the set of all representations $\rho: \Gamma \to \PSL_2\BC $ with discrete, torsion free and non-elementary image.  Here, a discrete subgroup of $\PSL_2\BC $ is called elementary if it is virtually abelian.  The set $\calD (\Gamma) $ sits naturally in the variety $\Hom (\Gamma,\PSL_2\BC) $ and inherits the topology given by pointwise convergence; this is called the {\em algebraic topology} on $\calD (\Gamma) $ and a pointwise convergent sequence of representations is usually called {\em algebraically convergent}.  The goal of this note is to investigate the relationship between the algebraic convergence of a sequence $(\rho_i) $ in $\calD (\Gamma) $ and the {\em geometric convergence} of the subgroups $\rho_i (\Gamma) \subset \PSL_2\BC $.  Recall that a sequence of closed subgroups $(G_i)$ of $\PSL_2\BC$ converges to a subgroup $G \subset \PSL_2\BC $ geometrically if it does in the Chabauty topology.

Assume from now on that $(\rho_i)$ is a sequence in $\calD(\Gamma)$ converging algebraically to a representation $\rho\in \calD (\Gamma) $, and that the groups $\rho_i(\Gamma)$ converge geometrically to a subgroup $G$ of $\PSL_2\BC$.  For convenience, we will often say that $\rho_i $ converges geometrically to $G $.  
While it is clear that $\rho(\Gamma)\subset G$, it was Jorgensen \cite{Jorgensen1} who first realized that the geometric limit may be larger than the algebraic limit.  For example, Thurston \cite{Thu86B} constructed an algebraically convergent sequence in $\calD (\pi_1 (\Sigma_g)) $ that converges geometrically to a subgroup of $\PSL_2\BC$ that is not even finitely generated.  Other examples of this phenomenon, each one dramatic in its own way, were constructed by Kerckhoff-Thurston \cite{KT}, Anderson-Canary \cite{AC-book} and Brock \cite{Brock}.

All these examples are related to the appearance of new parabolic elements in the algebraic limit; in fact, the following holds:

\begin{sat}[Anderson-Canary\cite{AC1}]\label{AC}
Let $\Gamma $ be a finitely generated group and assume that $(\rho_i) $ is a sequence of faithful representations in $\calD (\Gamma) $ converging algebraically to some $\rho \in \calD (\Gamma)   $.  If $\rho(\Gamma)$ does not contain parabolic elements, then the groups $\rho_i(\Gamma)$ converge geometrically to $\rho(\Gamma)$.
\end{sat}

In \cite{AC2}, Anderson and Canary extended this result to the case where $\rho $ and $\rho_i$ map the same elements to parabolics for all $i $.  Evans proved in \cite{Evans} that the same conclusion holds under the weaker assumption that if an element of $\Gamma $ is sent to a parabolic by $\rho $ then it is also parabolic in $\rho_i $ for all $i$.  All these results were obtained in the presence of certain technical assumptions rendered unnecessary by work of Brock and Souto \cite{BS}, and a fortiori by the resolution of the tameness conjecture by Agol \cite{Agol} and Calegari-Gabai \cite{Calegari-Gabai}.
\vspace{0.2cm}

Motivated by questions of a different nature, related to the attempt to understand the structure of closed hyperbolic 3-manifolds whose fundamental group can be generated by, say, $10$ elements, we revisited Theorem \ref{AC} convinced that it would remain true after dropping the assumption that the representations $\rho_i$ are faithful. To our surprise, we found the following examples showing that Theorem \ref{AC} fails dramatically in this more general setting.

\begin{bei}\label{ex1}
Let $\Gamma$ be the fundamental group of a closed surface of genus $3$. There is a sequence of representations $ (\rho_i) $ in $\calD (\Gamma) $ converging algebraically to a faithful representation $\rho$ and geometrically to a subgroup $G \subset \PSL_2\BC $, such that 
\begin{itemize}
\item $G$ does not contain any parabolic elements.
\item $\rho(\Gamma)$ has index $2$ in $G$.
\end{itemize} 
\end{bei}

\begin{bei}\label{ex2}
Let $\Gamma$ be the fundamental group of a compression body with exterior boundary of genus $4$ and connected interior boundary of genus $3$. There is a sequence of representations $(\rho_i) $ in $\calD (\Gamma) $ converging algebraically to a faithful representation $\rho$ and geometrically to a group $G$, such that \begin{itemize}
\item $G$ does not contain any parabolic elements.
\item $\rho(\Gamma)$ has infinite index in $G$.
\end{itemize} 
\end{bei}

\begin{bei}\label{ex3}
Let $\Gamma$ be the fundamental group of a compression body with exterior boundary of genus $4$ and connected interior boundary of genus $3$. There is a sequence of representations $(\rho_i) $ in $\calD (\Gamma) $ converging algebraically to a faithful representation $\rho$ and geometrically to a group $G $, such that
\begin{itemize}
\item $\rho(\Gamma)$ does not contain any parabolic elements.
\item $G$ is not finitely generated.
\end{itemize} 
\end{bei}

Apart from discussing the examples above, our main goal is to understand the failure of the Theorem \ref{AC} and see what is still true without the assumption that the representations are faithful. We prove:

\begin{sat}\label{main}
Let $\Gamma$ be a finitely generated group and $(\rho_i)$ a sequence in $\calD(\Gamma)$.  Assume that $ (\rho_i) $ is algebraically convergent and converges geometrically to a subgroup $G$ of $\PSL_2\BC$. If $G$ does not contain parabolic elements, then $G$ is finitely generated.
\end{sat}

Before going further, observe that Example \ref{ex3} shows that for the conclusion of Theorem \ref{main} to hold it does not suffice to assume that the algebraic limit has no parabolics. Similarly, Example \ref{ex2} shows that under the assumptions of Theorem \ref{main} the algebraic limit $\rho(\Gamma)$ can have infinite index in the geometric limit $G$.

The reader may find it surprising that we mention Example \ref{ex1} at all; the other two seem to be much more dramatic. However, Example \ref{ex1} is {\em the mother of all examples}. It also shows that the following theorem of Anderson fails if one considers non-faithful representations:

\begin{sat}[Anderson]\label{anderson}
Assume that $\Gamma $ is a finitely generated group, and that $(\rho_i) $ is a sequence of faithful representations in $\calD (\Gamma) $ converging algebraically to some representation $\rho$ and geometrically to a subgroup $G$ of $\PSL_2\BC$. Then maximal cyclic subgroups of $\rho(\Gamma)$ are maximal cyclic in $G$. In particular, if the image $\rho(\Gamma)$ of the algebraic limit has finite index in the geometric limit $G$, then $\rho(\Gamma)=G$.
\end{sat}

We mention Theorem \ref{anderson} because its failure is the heart of the failure of Theorem \ref{AC} for sequences which are not necessarily faithful:

\begin{sat}\label{max-cyclic}
Let $\Gamma$ be a finitely generated group and $(\rho_i)$ a sequence in $\calD (\Gamma) $.  Assume that $ (\rho_i) $ converges algebraically to a representation $\rho$ and geometrically to a subgroup $G$ of $\PSL_2\BC$. If 
\begin{itemize}
\item $\rho(\Gamma)$ does not contain parabolic elements, and 
\item maximal cyclic subgroups of $\rho(\Gamma)$ are maximal cyclic in $G$,
\end{itemize}
then $G=\rho(\Gamma)$.
\end{sat}

The proof of Theorem \ref{max-cyclic}, as the proof of all results in this note, is completely independent of Theorem \ref{AC}. In fact, combining Theorem \ref{max-cyclic} with Anderson's Theorem \ref{anderson} we obtain a new simpler proof of Theorem \ref{AC}, which we will use in constructing Examples 1-3. It should be observed that while Anderson-Canary and Evans used fairly involved arguments to bypass the question of tameness, we use here the resolution of the tameness conjecture by Agol \cite{Agol} and Calegari-Gabai \cite{Calegari-Gabai} in a crucial way.
\vspace{0.2cm}

To conclude this fairly long introduction, we describe the paper section by section. After some preliminaries in section \ref{sec:preli} we construct the examples mentioned above in section \ref{sec:examples}. In section \ref{sec:max-cyclic} we prove Theorem \ref{max-cyclic}, obtaining a new proof of Theorem \ref{AC}. In section \ref{sec:attaching-roots} we obtain a technical result, Proposition \ref{attaching-roots}, which asserts that under some conditions if $M\to N$ is a cover then there is a tower $M\to M'\to N$ where $|\chi(M')|<|\chi(M)|$. Proposition \ref{attaching-roots} and a slightly stronger version of Theorem \ref{max-cyclic} are the key ingredients in the proof of Theorem \ref{main} in section \ref{sec:main}: the idea is to show that if $H\subset G$ is finitely generated, contains $\rho(\Gamma)$ and is such that $|\chi(\BH^3/H)|$ is minimal then $G=H$ and hence $G$ is finitely generated. In section \ref{sec:cusps} we discuss extensions of Theorem \ref{main} and Theorem \ref{max-cyclic} to the case that the algebraic limit has cusps; for instance, this permits us to recover Evans' general version of Theorem \ref{AC}. Once this is done, we discuss briefly in appendix A to which extent other well-known theorems about faithful representations remain true if the condition of faithfulness is dropped. \ref{sec:examples}.
\vspace{0.2cm}

\noindent \bf Acknowledgments: \rm  The authors would like to thank Misha Kapovich for the idea of using dimensions of deformation spaces to prove Proposition \ref{attaching-roots}.  We also  owe a great debt to the referee, who has greatly improved the readability of the paper.

\section{Preliminaries}\label{sec:preli}
In this section we recall some well-known facts and definitions about hyperbolic 3-manifolds. 

\subsection{Hyperbolic manifolds}
By a hyperbolic 3-manifold we will mean any Riemannian 3-manifold isometric to $\BH^3/\Gamma$, where $\Gamma$ is a discrete, torsion-free group of isometries of hyperbolic 3-space.  We will usually assume that the elements of $\Gamma $ are orientation preserving, or equivalently that $\BH^ 3/\Gamma $ is orientable.  The full group of orientation preserving isometries of hyperbolic $3 $-space is written $\Isom_+(\BH^3) $, and is often identified with $\PSL_2 \BC $ through its action on the boundary of $\Hyp^3 $.  

Conjugate subgroups of $\PSL_2\BC $ give isometric quotients of $\Hyp^3 $; in order to remove this indeterminacy we consider {\em pointed} hyperbolic 3-manifolds, i.e. pairs $(M,\omega)$ where $\omega$ is an orthonormal frame of some tangent space of $M$. Choosing once and for ever a fixed frame $\omega_{\BH^3}$ of some tangent space of $\BH^3$, every quotient manifold $\BH^3/\Gamma$ has an induced framing $\omega_{\BH^3/\Gamma}$ given by the projection of $\omega_{\BH^3}$. Now, if $(M,\omega)$ is a pointed hyperbolic 3-manifold then there is a unique $\Gamma\subset\PSL_2\BC$ such that the manifolds $(M,\omega)$ and $(\BH^3/\Gamma,\omega_{\BH^3/\Gamma})$ are isometric as pointed manifolds. 

\begin{bem}
It would be more natural to speak of {\em framed} hyperbolic 3-manifolds instead of pointed; however, it is customary to use the given terminology.
\end{bem}

\subsection{Tameness}
We will be mainly interested in hyperbolic 3-manifolds $M$ with finitely generated fundamental group. Any such manifold is {\em tame} by the work of Agol \cite{Agol} and Calegari-Gabai \cite{Calegari-Gabai}:

\begin{named}{Tameness Theorem}[Agol, Calegari-Gabai]
Let $M$ be a hyperbolic 3-manifold with finitely generated fundamental group. Then $M$ is homeomorphic to the interior of a compact 3-manifold.
\end{named}

A {\em compact core} of a hyperbolic 3-manifold $M$ with finitely generated fundamental group is a compact submanifold $C\subset M$ with $M\setminus C$ homeomorphic to $\D C\times\BR$. Observe that $M$ is homeomorphic to the interior of every such compact core. It follows from the Tameness Theorem that every such $M$ admits an exhaustion by nested compact cores. If $C\subset M$ is a standard compact core, then the ends of $M$ correspond naturally to components of $M\setminus C$. The component $U_\CE$ of $M\setminus C$ corresponding to an end $\CE$ is said to be a {\em standard neighborhood} of $\CE$ and the component of $\D C$ contained in the closure of $U_\CE$ is said to {\em face} $\CE$. We will often denote the component of $\D C$ facing $\CE$ by $\D\CE$. Observe that $U_\CE$ is homeomorphic to $\D\CE\times\BR$.

\subsection{Geometry of ends in the absence of cusps}
\label{ends}
Let $M$ be a hyperbolic 3-manifold with finitely generated fundamental group and without cusps.  An end $\CE$ of $M$ is {\em convex cocompact} if it has a neighborhood in $M$ disjoint from the convex-core $CC(M)$ of $M$. Recall that the convex core $CC(M)$ is the smallest convex submanifold of $M$ whose inclusion is a homotopy equivalence. A manifold with compact $CC(M)$ is said to be {\em convex cocompact}; equivalently, all ends of $M$ are convex cocompact.

For every $d>0$, the set of points in $M$ within distance $d$ of $CC(M)$ is homeomorphic to $M$ and has \it strictly convex \rm $C^1$-boundary: that is, $\langle\DD_X\nu,X\rangle>0$, where $\DD$ is the Levi-Civita connection, $X $ is a vector tangent to $\partial K_i $ and $\nu$ is the outer normal field along $\D K_i$.  

Smoothing the boundary, we obtain the following well-known fact:

\begin{lem}\label{exhaustion-convex}
Let $M$ be a hyperbolic 3-manifold with finitely generated fundamental group. There is an exhaustion of $M$ by a nested sequence of submanifolds $ K_i $ such that:
\begin{enumerate}
\item The boundary $\D K_i$ is smooth and strictly convex.  
\item Every convex cocompact end of $M$ has a neighborhood disjoint of $K_i$.
\item The inclusion of $K_i$ into $M$ is a homotopy equivalence.\qed
\end{enumerate} 
\end{lem}

Continuing with the same notation as in Lemma \ref{exhaustion-convex}, convexity implies that there is a well-defined map $\kappa_{K_i}:M\to K_i$ that takes a point in $M$ to the point in $K_i $ closest to it.  Strict convexity implies that the preimage of a point $x\in\D K_i$ under this projection is a geodesic ray.  It follows that the map
$$M\setminus K_i\to\D K_i\times(0,\infty),\ \ x\mapsto(\kappa_{K_i}(x),d(x,\kappa_{K_i}(x))$$
is a diffeomorphism; in fact, its inverse is the radial coordinate map
$$\partial K_i\times (0,\infty) \to M \setminus K_i, \ \ (x, t) \mapsto \exp_x (t \nu (x)), $$ where $\nu $ is the outer unit normal vector-field along $\partial K_i $.

An end $\CE$ which is not convex cocompact is said to be {\em degenerate}. It follows from the Tameness Theorem and earlier work of Bonahon \cite{Bonahon} and Canary \cite{Canary-ends} that degenerate ends have very well-behaved geometry. For instance, every degenerate end $\CE$ has a neighborhood which is completely contained in the convex core $CC(M)$. From our point of view, the most important fact about degenerate ends is the Thurston-Canary \cite{Canary-covering} Covering Theorem, of which we state the following weaker version:

\begin{sat}[Thurston, Canary \cite {Canary-covering}]\label {thecoveringtheorem}
Let $M$ and $N$ be non-compact hyperbolic 3-manifolds, assume that $M$ has finitely generated fundamental group and no cusps, let $\pi:M\to N$ be a Riemannian cover and let $\CE$ be a degenerate end of $M$. Then $\CE$ has a standard neighborhood $U_\CE$ such that the restriction
$$\pi\vert_{U_{\CE}}:U_{\CE}\to\pi(U_{\CE})$$
of the covering $\pi$ to $U_{\CE}$ is a covering map onto a standard neighborhood of a degenerate end $\CE'$ of $N$. More precisely, there is a finite covering $\sigma:\D\CE\to\D\CE'$ and homeomorphisms
$$\phi:\D\CE\times\BR\to U_\CE,\ \ \psi:\D\CE'\times\BR\to\pi(U_\CE)$$
with $(\psi^{-1}\circ\pi\vert_{U_\CE}\circ\phi)(x,t)=(\sigma(x),t)$. In particular, the covering $\pi\vert_{U_\CE}$ has finite degree.
\end{sat}

Combining the Tameness theorem, Lemma \ref{exhaustion-convex} and the Covering theorem we obtain:

\begin{prop}\label{exhaustion}
Let $M$ and $N$ be hyperbolic 3-manifolds with infinite volume, assume that $M$ has finitely generated fundamental group and no cusps, and let $\pi:M\to N$ be a Riemannian cover. Then $M$ admits an exhaustion by nested standard compact cores $C_i\subset C_{i+1}$ such that the following holds:
\begin{enumerate}
\item If a component $S$ of $\D C_i$ faces a convex cocompact end of $M$ then $S$ is smooth and strictly convex.
\item If a component $S$ of $\D C_i$ faces a degenerate end of $M$ then the restriction $\pi\vert_S:S\to\pi(S)$ is a finite covering onto an embedded surface in $N$.\qed
\end{enumerate}
\end{prop}

\subsection{Conformal boundaries and Ahlfors-Bers theory}
\label{conformalboundaries}
As in the previous section, assume that we have a hyperbolic $3$-manifold $M=\Hyp^ 3/ \Gamma $ with finitely generated fundamental group and no cusps.  

The {\em limit set} of $\Gamma $, written $\Lambda (\Gamma) $, is the closure of the set of fixed points in $\BS^2_{\infty } = \partial \Hyp^ 3 $ of hyperbolic elements of $\Gamma $.  The complement of $\Lambda (\Gamma) $ is the {\em domain of discontinuity } $\Omega (\Gamma) = \BS^2_{\infty } \setminus \Lambda (\Gamma) $, which is the largest open subset of $\BS^2_{\infty } $ on which $\Gamma $ acts properly discontinuously.  The quotient $\partial_c M = \Omega (\Gamma)/ \Gamma $ is called the \it conformal boundary \rm of $M $.  In fact, $\Gamma $ acts properly discontinuously on $\Hyp^3 \cup \Omega (\Gamma) $, and the quotient $M \cup \partial_c M  $ is a manifold with boundary having interior $ M $.  

The action of $\Gamma $ on $\BS^2_{\infty } $ is by Mobius transformations, so $\partial_c M $ inherits a natural conformal structure.  This structure is closely tied with the geometry of $M$: for instance, the unique hyperbolic metric on $\partial_c M $ compatible with this conformal structure, called the {\em Poincar\'e metric}, is similar to the intrinsic metric on $\partial CC( M) $.  Specifically, the closest point projection $\kappa: M \to CC (M) $ extends continuously to a map $\bar {\kappa } :\partial_c M \to \partial CC( M) $, and we have the following theorem of Canary: 

\begin{prop}[Canary \cite{Canary-conformal}]
For every $\epsilon >0 $ there exists $K > 0 $ so that the following holds.
Let $M$ be a hyperbolic $3$-manifold with finitely generated fundamental group, such that every component of $\partial_c M $ has injectivity radius at least $\epsilon $ in the Poincar\'e metric.  Then the closest point projection $\kappa : \partial_c M\to \partial CC (M) $ is $K $-lipschitz, where $\partial _c M $ has the Poincar\'e metric and $\partial CC (M) $ is considered with the Riemannian metric induced by its inclusion into $M$.  \label{conformalconvex}
\end{prop}

The conformal boundary plays an important role in the deformation theory of hyperbolic $3$-manifolds; in particular, a convex-cocompact hyperbolic $3$-manifold is determined up to isometry by its topology and conformal boundary.  A more precise statement of this is as follows. 

 Let $M$ be the interior of a compact hyperbolizable $3$-manifold $\bar M $ in which each boundary component has negative Euler characteristic.  Define $\calcium (M) $ to be the set of all convex-cocompact hyperbolic metrics on $M $, where two metrics are identified if they differ by an isometry isotopic to the identity map.  It follows from Thurston's hyperbolization theorem \cite{Thu86A} that $\calcium (M) $ is nonempty, and it inherits a natural complex structure through its relation to the representation variety $\Hom (\pi_1(M),\PSL_2\BC) $ (see \cite[Section 4.3]{Japaner}).  Then we have:

\begin{sat}[Ahlfors-Bers Parameterization, see \cite{Japaner}] 
\label{AhlforsBers}

The map $\calcium (M) \to \CT (\partial\bar{ M }) $, induced from the map taking a convex-cocompact uniformization of $M$ to its conformal boundary, is a biholomorphic equivalence.  Therefore, $\calcium (M) $ is a complex manifold of dimension $$ \frac 32 |\chi (\partial\bar { M }) |=3 |\chi (M)| .$$
\end{sat}
Here, $\CT (\partial \bar M) $ is the \it Teichm\"uller space \rm of conformal structures on $\partial \bar M $, where two conformal structures are identified if there is a conformal map between them that is isotopic to the identity.  The conformal boundary of a convex-cocompact uniformization of $M $ is only identified with $\partial \bar M $ up to isotopy, so a point of $\calcium (M) $ gives a point in $\CT (\partial\bar M) $ rather than a specific conformal structure on $\partial \bar M $.

The space $\calcium (M) $ is natural with respect to certain coverings.  If $M' $ is a cover of $M $ with $\pi_1 (M') $ finitely generated, then the Tameness Theorem and Canary's covering theorem imply that convex-cocompact metrics on $M$ lift to convex-cocompact metrics on $M' $.  In fact,
  
\begin{lem}
If $M' $ is a $3$-manifold with finitely generated fundamental group and $\tau:  M' \to M $ is a covering map, there is a holomorphic map $$\tau^* : \calcium (M) \to \calcium (M') $$ induced by the map taking a hyperbolic structure on $M $ to its pullback under $\tau $.
\end{lem}

\subsection{Geometric convergence}
Recall that a sequence $(G_i)$ of closed subgroups of $\PSL_2\BC$ converges {\em geometrically} to a subgroup $G$ if it does in the Chabauty topology. More concretely, $(G_i)$ converges geometrically to $G$ if $G$ is the subgroup of $\PSL_2\BC$ consisting precisely of those elements $g\in\PSL_2\BC$ such that there are $g_i\in G_i$ with $g_i\to g$ in $\PSL_2\BC $. In other words, $G$ is the accumulation set of the groups $G_i$.

Most of our arguments are based on an interpretation of geometric convergence in terms of the quotient manifolds $\Hyp^ 3 / G_i $. 

\begin{defi*}
A sequence $(M_i,\omega_i)$ of pointed hyperbolic 3-manifolds converges {\em geometrically} to a pointed manifold $(M_\infty,\omega_\infty)$ if for every compact $K\subset M_\infty$ that contains the origin of $\omega_\infty$, there is a sequence $\phi_i:K\to M_i$ of smooth maps with $\phi_i(\omega_\infty)=\omega_i$ converging in the $C^k$-topology to an isometric embedding for all $k\in\BN$. We will refer to the maps $\phi_i$ as the {\em almost isometric embeddings provided by geometric convergence}. 
\end{defi*}

\begin{rmk*}
Note that although the phrase `converging in the $C^k $-topology to an isometric embedding' is suggestive and pleasing to the ear, it has no meaning.  One way to formalize this would be to say that for each point $x \in K $, there is an $\epsilon > 0 $ and a sequence of isometric embeddings $$\beta_i : B ( \phi_i (x),\epsilon) \to \Hyp^3 $$ from $\epsilon $-balls around $\phi_i (x) \in M_i $ so that $\beta_i \circ \phi_i $ converges to an isometric embedding of some neighborhood of $x \in M_\infty $ into $\Hyp^ 3 $. 
\end{rmk*}

Recall that by our convention above, a pointed hyperbolic manifold is a manifold together with a base frame and that choosing a base frame $\omega_{\BH^3}$ of hyperbolic space we obtain a bijection between the sets of discrete torsion free subgroups of $\PSL_2\BC$ and of pointed hyperbolic 3-manifolds. Under this identification, the notions of geometric convergence of groups and manifolds are equivalent (see for instance \cite{Benedetti-Petronio,Kapovich}):

\begin{prop}\label{convergence-convergence-geom}
Let $G_1,G_2,\dots,G_\infty$ be discrete and torsion-free subgroups of $\PSL_2\BC$ and consider for all $i=1,\dots,\infty$ the pointed hyperbolic 3-manifold $(M_i,\omega_i)$ where $M_i=\BH^3/G_i$ and $\omega_i$ is the projection of the frame $\omega_{\BH^3}$ of $\BH^3$. The groups $G_i$ converge geometrically to $G_\infty$ if and only if the pointed manifolds $(M_i,\omega_i)$ converge geometrically to $(M_\infty,\omega_\infty)$.
\end{prop}

\subsection{Algebraic convergence}

Let $\Gamma$ be a finitely generated group. Recall that a sequence $(\rho_i)$ of representations $\rho_i:\Gamma\to\PSL_2\BC$ converges {\em algebraically} to a representation $\rho$ if for every $\gamma\in\Gamma$ we have $\rho_i(\gamma)\to\rho(\gamma)$ in $\PSL_2\BC$.  Jorgensen proved in \cite{Jorgensen} that if each image $\rho_i (\Gamma) $ is discrete and non-elementary then the same is true of $\rho (\Gamma) $.  His argument also shows that torsion cannot suddenly appear in the limit, so we have the following theorem:

\begin{prop}\label{Jorgensen}
Let $\Gamma$ be a finitely generated group.  Then the subset $\calD (\Gamma) \subset \Hom (\Gamma,\PSL_2\BC) $ consisting of (not necessarily faithful) representations with discrete, torsion free and non-elementary image is closed with respect to the algebraic topology.
\end{prop}

When a sequence $ (\rho_i) $ converges both algebraically to a representation $\rho$ and geometrically to some group $G$, it is easy to see that $\rho(\Gamma) \subset G$. In other words, the manifold $\BH^3/\rho(\Gamma)$ covers the manifold $\BH^3/G$. In particular, given a compact subset $C\subset\BH^3/\rho(\Gamma)$ we can project it down to $\BH^3/G$ and then map the image to $M_i $ under the almost isometric embeddings given by geometric convergence.  This produces maps $C\to M_i=\BH^3/\rho_i(\Gamma)$ which look more and more like the restriction of a covering to $C$. 

More generally, assume that $H$ is a finitely generated subgroup of $G$ containing $\rho(\Gamma)$. By the tameness theorem, the manifold $\BH^3/H$ contains a standard compact core $C_H$. Composing the restriction to $C_H$ of the covering $\BH^3/H\to\BH^3/G$ with the almost isometric embeddings given by geometric convergence, we obtain for sufficiently large $i$ maps $C_H \to M_i $ similar to those described above.  Using the induced homomorphisms $H\to\pi_1(M_i,\omega_i)$ one can then construct a sequence of representations $\sigma_i:H\to\PSL_2\BC$ converging algebraically to the inclusion of $H\hookrightarrow G\hookrightarrow\PSL_2\BC$. The assumption that $\rho(\Gamma)\subset H$ implies then that $\sigma_i(H)=\rho_i(\Gamma)$ for all $i$. In particular, (compare with \cite[Lemma 4.4]{McMullen}):

\begin{prop}\label{competitors}
Let $\Gamma$ be a finitely generated group, $(\rho_i)$ a sequence in $\calD (\Gamma) $ converging algebraically to a representation $\rho_\infty$ and geometrically to a group $G \subset \PSL_2\BC $. If $H\subset G$ is a finitely generated subgroup of $G$ containing $\rho(\Gamma)$ then for large $i $ there are representations $$\sigma_i:H\to\PSL_2\BC, \text { with } \sigma_i(H)=\rho_i(\Gamma)$$ converging to the inclusion $H\hookrightarrow\PSL_2\BC$. In particular, the groups $\sigma_i(H)$ converge geometrically to $G$.\qed
\end{prop}

\subsection{Roots}
Recall that the nontrivial elements in $\PSL_2\BC$ are either hyperbolic, parabolic or elliptic depending on their dynamical behaviour. Every hyperbolic element $\gamma\in\PSL_2\BC$ stabilizes a geodesic $A$ in $\BH^3$, and if $\alpha\in\PSL_2\BC$ is a $k$-th root of $\gamma$, i.e. $\gamma=\alpha^k$, then $\alpha A=A$.  It follows easily that the set of $k$-th roots of $\gamma$ is finite.  A similar argument applies in the parabolic and elliptic case, so we obtain the following well-known, and in this paper surprisingly important, fact:

\begin{lem}\label{roots}
For every $k\in\BZ$ and nontrivial element $\gamma\in\PSL_2\BC$, the set $\{\alpha\in\PSL_2\BC \ \vert \ \alpha^k=\gamma\}$ is finite.\qed
\end{lem}

Lemma \ref{roots} has the following immediate consequence:

\begin{kor}\label{extension-finite}
Let $\Gamma\subset\Gamma'$ be groups and assume that $\Gamma'$ contains a torsion-free subgroup $H$ such that 
\begin {enumerate}
\item $\Gamma$ and $H$ generate $\Gamma'$, and
\item $\Gamma\cap H$ has finite index in $H$.
\end{enumerate}
 Then for every faithful representation $\rho:\Gamma\to \PSL_2\BC$, the set of representations $\{\rho':\Gamma'\to\PSL_2\BC \ \vert \ \rho'\vert_\Gamma=\rho\}$ is finite.\qed
\end{kor}

 Lemma \ref{roots} also has the following consequence, which was stated in the introduction.

\begin{named}{Theorem \ref{anderson}}[Anderson]
Assume that $\Gamma $ is a finitely generated group and $(\rho_i)$ is a sequence of faithful representations in $\calD (\Gamma) $ that converges algebraically to a representation $\rho $ and geometrically to a group $G $.  Then maximal cyclic subgroups of $\rho(\Gamma)$ are maximally cyclic in $G$. In particular, if the image $\rho(\Gamma)$ of the algebraic limit has finite index in the geometric limit $G$, then $\rho(\Gamma)=G$.
\end{named}
\begin{proof}
If some maximal cyclic subgroup of $\rho (\Gamma) $ is not maximally cyclic in $G $, then there is some $g\in G \setminus \rho (\Gamma) $ that powers into it.  So, $g^k =\rho(\eta) $ for some $\eta\in\Gamma$ and $k\ge 2$.  Since $g\in G$ there are $\gamma_i\in\Gamma$ with $\lim_{i \to \infty }\rho_i(\gamma_i)=g$. Taking powers we have then
$$\lim_{i \to \infty }\rho_i(\gamma_i^k)=g^k=\rho (\eta) =\lim_{i \to \infty }\rho_i(\eta).  $$
It follows from the Margulis lemma that $\rho_i(\gamma_i^k)=\rho_i(\eta)$ for sufficiently large $i$.  \it Since the representations $\rho_i$ are faithful, \rm this implies that $\gamma_i^k=\eta$ for large $i$.  But the group $\Gamma$ embeds in $\PSL_2\BC $, by say $\rho_7 $, so Lemma \ref{roots} shows that each of its elements has only finitely many $k^{th } $-roots.  This implies that up to passing to a subsequence we may assume that $\gamma_i=\gamma_j$ for all pairs $ (i,j) $, and hence $g$ belongs to the algebraic limit.  This is a contradiction, so the claim follows.  
\end{proof}

We included this proof because its failure for non-faithful sequences is at the heart of the examples presented in Section \ref{sec:examples}.

\subsection{Maximal Cusps}
\label{maximalcusps}
Finally, we describe a class of hyperbolic manifolds that we will use as building blocks in constructing our examples in Section \ref{sec:examples}.  A good reference for this section is \cite{Canary-Culler-Hersonsky-Shalen}.

Assume that $ M $ is a compact, orientable, irreducible and atoroidal $3$-manifold with interior $ N $ and no torus boundary components.  If $N $ has a geometrically finite hyperbolic metric, then there is a collection $C $ of disjoint simple closed curves in $\partial M $ such that $$N \cup \partial_c N \cong M \setminus \cup_{\gamma \in C} \gamma, $$ by a homeomorphism whose restriction to $N $ is isotopic to the inclusion $N \subset M $.  Here, $\partial_c N $ is the \it conformal boundary \rm defined in Section \ref {conformalboundaries}.

The curves in $C $ are determined up to isotopy, and correspond to the rank one cusps in $N $.  We will say that the collection $C $ has been \it pinched. \rm Each curve in $C $ is homotopically nontrivial in $M $ and no two curves in $C $ are freely homotopic in $M $.  It follows from Thurston's Hyperbolization Theorem \cite{Kapovich,Otal-Haken} that any collection of curves on $\partial M $ satisfying these two topological constraints can be obtained as above from a hyperbolic structure on $N $; it is therefore said that such a collection is \it pinchable.  \rm 

One says that a component $S \subset \partial M $ is \it maximally cusped \rm with respect to some geometrically finite structure on $N $ if the associated collection $C $ contains a pants decomposition for $S $.  In this case, any component of $\partial CC (N) $ that faces $S $ is a totally geodesic hyperbolic thrice-punctured sphere.  Given two hyperbolic $3$-manifolds with maximally cusped ends, we can topologically glue their convex cores together along these thrice-punctured spheres.  Since every homeomorphism between hyperbolic thrice-punctured spheres can be isotoped to an isometry, after altering the identifications we can ensure that our gluing produces a hyperbolic $3$-manifold.  Here is a precise description of the result of this process.

\begin{lem}[Gluing Maximal Cusps]
\label{gluingcusps}
Let $M_1, M_2$ be compact, orientable $3$-manifolds with interiors $N_i$ that have geometrically finite hyperbolic metrics.  Assume that $S_i $ are maximally cusped components of $\partial M_i $ with pinched pants decompositions $ P_i \subset S_i $.  Then if $h : S_1 \to S_2 $ is a homeomorphism with $h (P_1) = P_2 $, there is a hyperbolic $3$-manifold $N $ with the following properties:
\begin{itemize}
\item $N \cong ( M_1 \setminus P_1) \sqcup_h (M_2 \setminus P_2) $, so $N $ has a rank 2 Icusp corresponding to each element of $P_1 $ (or $P_2 $)
\item $N $ is the union of two subspaces with disjoint interiors, which are isometric to $N_1 \setminus E_1 $ and $N_2 \setminus E_2 $, where $E_i$ are the components of $N_i \setminus CC (N_i ) $ adjacent to $S_i $.

\end{itemize}
\end{lem}

\section{Examples}\label{sec:examples}

In this section we construct the examples mentioned in the introduction.  All our constructions follow the same general strategy.  We will describe this here before starting with the examples in detail.  

We first construct a finitely generated group $\hat\Gamma $ and a sequence of (unfaithful) discrete representations $ (\sigma _i ) $ in $\calD (\hat\Gamma) $ that converges \it strongly \rm to a faithful representation $\sigma_\infty$ whose image $\sigma_\infty(\hat\Gamma)$ has no parabolics.  Recall that $\sigma_i \to \sigma_\infty $ \it strongly \rm if it does so algebraically and $\sigma_i(\hat\Gamma)\to \sigma_\infty (\hat \Gamma)$ geometrically.  
Next, we find a proper subgroup $\Gamma\subset\hat\Gamma$ such that $$\sigma_i(\Gamma)=\sigma_i(\hat\Gamma) \text { for all } i\in\BN$$ and let $\rho_i=\sigma_i\vert_{\Gamma}$. By construction, $\rho_i \to \rho_\infty =\sigma_\infty |_{\Gamma } $ algebraically and $\rho_i (\Gamma) \to \sigma_\infty (\hat\Gamma) $ geometrically. Since $\Gamma$ is a proper subgroup of $\hat\Gamma$ and $\sigma_\infty$ is faithful, we obtain that the algebraic limit $\rho_\infty(\Gamma)=\sigma_\infty(\Gamma)$ is a proper subgroup of the geometric limit $\sigma_\infty(\hat\Gamma)$.  

\begin{named}{Example \ref{ex1}} 
Let $\Gamma$ be the fundamental group of a closed orientable surface of genus $3$. There is a sequence of representations $(\rho_i) $ in $\calD (\Gamma) $ that converges algebraically to a faithful representation $\rho $ and geometrically to a group $G $ such that
\begin{itemize}
\item $G $ has no parabolics
\item $\rho(\Gamma)$ has index $2$ in $G$.
\end{itemize} 
\end{named}

In the following, we let $S $ be a closed orientable surface of genus $g $.  The group $\hat \Gamma $ from the strategy above will be $\pi_1 S $, while $\Gamma $ will be the fundamental group of a degree $2 $ cover of $S $.  The concise statement of the example above is the special case $g= 2 $, but our argument below works for any closed orientable surface.

To begin with, let $H $ be a handlebody with boundary $S $.  Recall that the \it Masur domain \rm  $O_H \subset \PML (S) $ is the set of all projective measured laminations on $S $ that intersect positively with every element of $\PML (S) $ that is a limit of meridians.  One calls a pseudo-Anosov map $f: S\to S $ \it generic \rm if its attracting lamination $\lambda $ lies in $O_H $.  To see that such maps exist, note that $O_H $ is open and nonempty and that the attracting laminations of pseudo-Anosov maps are dense in $\PML (S) $.  The density follows from the density of simple closed curves, for conjugating any given pseudo-Anosov map by powers of a Dehn twist produces a sequence of pseudo-Anosov maps whose attracting laminations limit to the twisting curve in $\PML (S) $.

\vspace{2mm} \noindent{\bf A Sequence of Convex-Cocompact Handlebodies:} \rm Fix a generic pseudo-Anosov map $f: S \to S $ with attracting lamination $\lambda $ and repelling lamination $\bar \lambda $.  Theorem \ref{AhlforsBers} implies that the deformation space of convex-cocompact hyperbolic metrics on a hyperbolizable $3$-manifold is parameterized by the Teichmuller space of its boundary.  So, we can produce a sequence of convex-cocompact metrics on $H $ corresponding to an orbit of $f $ on $\CT (S) $.  In other words, after fixing a conformal structure $X$ on $S $, we have a sequence of convex-cocompact hyperbolic manifolds $(N_i) $ and a sequence of homeomorphisms $$h_i : (H,S) \to ( N_i ,\partial_c N_i) $$ such that $h_i \circ f^{i} : S \to \partial _c N_i $ is conformal with respect to $X$.  Here, $\partial_c N_i $ is the \it conformal boundary \rm of $N_i $ discussed in Section \ref {conformalboundaries}.

\vspace{2mm} \noindent {\bf Good Markings:} \rm  By Proposition \ref{conformalconvex} and the fact that the Poincar\'e metric on $\partial _c N_i $ is constant in moduli space, the nearest point projection $\eta_i : \partial _c { N_i } \to\partial CC (N_i) $ is $K $-lipschitz for some constant $K $ independent of $i $.  Considering $S $ with its Poincar\'e metric, the map $$\sigma_i : = \eta_i \circ h_i \circ f^i : S \to N_i $$ is $K $-lipschitz as well.  Since it is $\pi_1 $-surjective, we may use $\sigma_i $ to mark the fundamental group of $N_i $ with $\pi_1 (S) $.  Specifically, after taking a base point $ p \in S $, the surjections $ (\sigma_i)_* :  \pi_1 (S, p) \to\pi_1 (N_i,\sigma_i (p)) $ determine up to conjugacy a sequence of representations $$\rho_i : \pi_1 (S, p) \to \PSL (2,\BC), \, \,  \Hyp^3 / \rho_i (\pi_1 (S, p)) = N_i.  $$

\vspace{2mm} \noindent {\bf Strong convergence and the limit: } \rm The goal here is to establish the following proposition.
\begin{prop}
\label {characterization}
The sequence $(\rho_i) $ converges strongly to a representation $\rho_\infty : \pi_1 (S, p) \to \PSL (2,\BC) $ that is faithful and purely loxodromic. Its quotient $N_\infty = \BH^ 3/ \rho_\infty (\pi_1 (S, p)) $ is homeomorphic to $S \times \BR $ and has no cusps. One of the ends of $N_\infty $ is convex-cocompact with conformal boundary $X$ and the other is degenerate with ending lamination $\bar\lambda$.\label {findthelimit}
\end{prop}

Here, $X$ is the base conformal structure on $S$ fixed above and $\bar\lambda$ is the repelling lamination of $f$.  Recall that an algebraically convergent sequence $\rho_i \to \rho_\infty $ converges \it strongly \rm if in addition the images of $(\rho_i)$ converge to the image of $\rho_\infty $ geometrically.   

\begin{figure}
\label{characterizationfigure}
\includegraphics{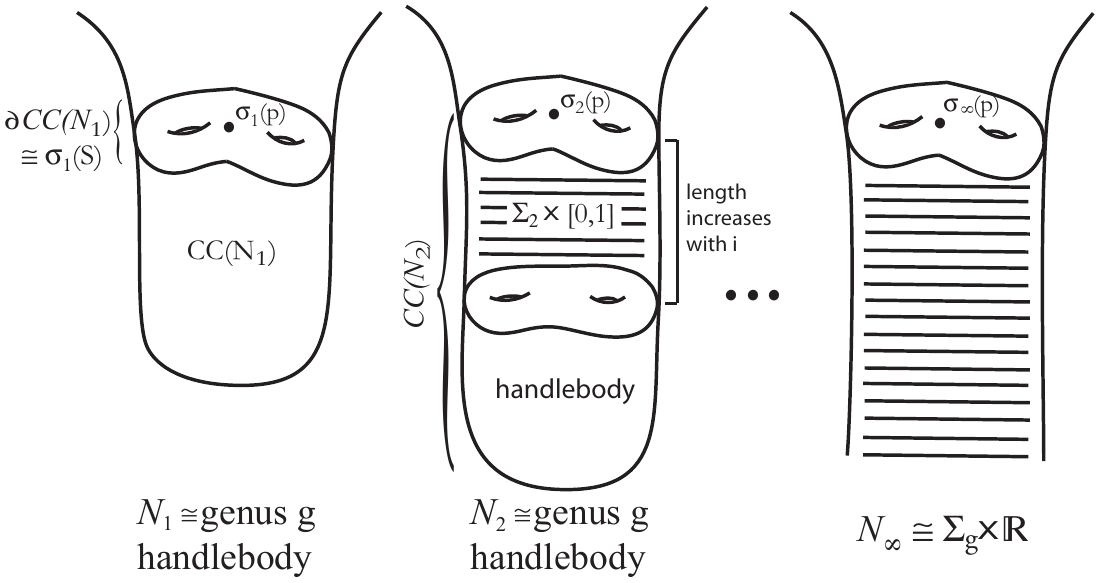}
\caption{A schematic of the convergence $N_i \to N_\infty $.}
\end{figure}

\vspace{2mm}

The resolution of Thurston's Ending Lamination Conjecture by Minsky \cite{ELC1} and Brock-Canary-Minsky \cite {ELC2} implies that $N_\infty $ is uniquely characterized by the description given in Proposition \ref {findthelimit}.   So to prove Proposition \ref {findthelimit}, it suffices to show that every subsequence of $(\rho_i) $ has a subsequence that converges strongly to some $\rho_\infty $ as described above.

Start with some subsequence $(\rho_{i_j}) $ of $(\rho_i )$. Passing to a further subsequence, we may assume that $(\rho_{i_j}) $ converges algebraically to some
$$\rho_\infty : \pi_1 (S, p) \to \PSL (2,\BC) , \text { where }  N_\infty : = \BH^ 3/ \rho_\infty (\pi_1 (S, p)) . $$ We may also assume that the images converge geometrically:$$\rho_{i_j} (\pi_1 (S, p)) \to G \subset \PSL (2,\BC) ,\text { where } N_G :=\Hyp^3 / G.$$ If base-frames are chosen for $N_{i_j} $ on $\partial CC (N_{i_j}) $, the resulting pointed manifolds converge in the Gromov Hausdorff topology to $N_G $.  By Arzela-Ascoli's Theorem, we can pass to a final subsequence such that the $K $-lipschitz maps $\sigma_{i_j} : S \to N_{i_j}$ converge uniformly to some
$$ \sigma_G : S \longrightarrow N_G . $$  There is a natural covering map $\pi : N_\infty \to N_G $ corresponding to the inclusion $\rho_\infty (\pi_1 (S, p)) \subset G $, and $\sigma_G $ lifts to a map $\sigma_{\rho } : S \to N_\infty $ inducing the marking given by $\rho_\infty $.  

\begin{claim}
$\sigma_\infty $ is an embedding, and its image bounds a neighborhood $E' $ of a convex-cocompact end of $N_\infty $.  Furthermore, $E' \cong \Sigma_g \times \BR $ and $\pi |_{E' } : E' \to N_G $ is an embedding.  
\end{claim}
\begin{proof}
We will show that the image of $\sigma_G $ bounds a convex-cocompact end in $N_G $ homeomorphic to $\Sigma_g \times \BR $.  It will follow immediately that this end lifts to an end $E' $ as desired.  

Let $K \subset N_G $ be a Gromov-Hausdorff limit of the sequence of convex cores $CC(N_{i_j}) $, passing to another subsequence as necessary.  Note that $K $ is convex and contains $CC (N_G) $.  Its boundary $\partial K $ is the limit of the boundaries $\partial CC(N_{i_j}) $, and is therefore the image of $\sigma_G $.  Using convexity, it is then not hard to see that $\sigma_G $ is an embedding.  Setting $E = N_G \setminus K $, we obtain a convex-cocompact end of $N_G $.  As $\partial E \cong \Sigma_g $, the nearest point retraction gives a homeomorphism $N_G \setminus K \cong \Sigma_g \times (0,\infty) $.  \end{proof}

The following claim is the reason our pseudo-Anosov map $f : S \to S $ was chosen to have attracting lamination $\lambda $ in the Masur domain $O_H $.  

\begin{claim} 
$\rho_\infty $ is faithful.
\end{claim}
\begin{proof}
An equivalent statement  is that the embedding $\sigma_\infty : S \to N_\infty $ is $\pi_1 $-injective.  So by the Loop Theorem, it suffices to check that $\rho_\infty $ is injective on elements of $\pi_1 (S, p) $ representable by simple loops on $S $.

Let $\gamma \subset S $ be a simple closed curve.  Then $f^{i_j} (\gamma) \to \lambda $ in $\PML $; since $\lambda \in O_H $, which is an open subset of $\PML $, for large $i_j$ we have $f^{i_j} (\gamma) \in O_H $ as well.  In particular, $ f^{i_j} (\gamma) $ is not compressible in $H $.  It follows that $$\sigma_{i_j} (\gamma) = \eta_{i_j} \circ h_{i_j} \circ f^{i_j} (\gamma) $$ is incompressible in $N_{i_j} $.  Therefore $\rho_{i_j} (\gamma) \neq \Id $ for sufficiently large $i_j $.  A standard application of the Margulis Lemma shows that $\rho_\infty (\gamma) \neq \Id $ as well (see, e.g.\cite[Theorem 7.1]{Japaner}).  
\end{proof}

Geometrically, the reason that $\rho_\infty $ is faithful is that for large $i $ all compressible curves on $\partial CC (N_i) $ are very long.  Since a fixed generating set for $\pi_1 (S, p) $ maps under our markings to a set of loops on $\partial CC (N_i) $ with uniformly bounded length, this implies that elements of $\pi_1 (S, p) $ representing compressible curves in $N_i $ must for large $i $ be expressible only with very large words in the generators.  So in the limit, there are no compressible curves.

As $\rho_\infty :\pi_1 (S, p) \to \PSL (2,\BC) $ is faithful, the Tameness Theorem implies that $N_\infty \cong S\times \BR $.  From above, we know that one of the two topological ends of $N_\infty $ is convex-cocompact with associated conformal structure $X$.  To analyze the geometry of the other end, we must first prove the following:

\begin{claim}
The lamination $\bar\lambda$ is not realized in $N_\infty$. In particular, $N_\infty$ has a degenerate end with ending lamination $\lambda$.
\end{claim}
\begin{proof}
Fix a meridian $m$. By definition, the curve $f^{-i_j}(m)$ is in the kernel of the representation $\rho_{i_j}$ for all $i_j$. On the other hand, the sequence $(f^{-i_j}(m))$ converges to $\bar\lambda$ in $\CP\CM\CL(S)$. This implies that $\bar\lambda$ is not realized; see for instance \cite[Section 4]{Hossein-Souto}.
\end{proof}

Since $N_\infty$ is homeomorphic to $S\times\BR$ and has a convex-cocompact end and a degenerate end, we deduce that $\rho_\infty$ is purely loxodromic. 

\begin{claim}
$\rho_i $ converges to $\rho_\infty $ strongly.
\end{claim}

\begin{proof}
We need to show that the covering $\pi :N_\infty \to N_G $ is trivial.  There is a neighborhood of the degenerate end of $N_\infty $ that is completely contained in the convex core $CC (N_\infty ) $.  So, $N_\infty \setminus E' $ is the union of $CC (N_\infty ) $ and some compact set.  A result of Thurston \cite{KT} and Bonahon \cite{Bonahon} implies that the injectivity radius of $N_\infty $ is bounded above inside of its convex core, so by extension there is an upper bound $K $ for the injectivity radius of $N_\infty $ outside of $E' $.  Pick a point $x \in N_G $ deep enough inside its convex-cocompact end so that $\inj (x,N_G) > K $.  Then no element of $N_\infty \setminus E' $ can project to $x $.  Since $\pi $ is injective on $E' $, $ |\pi^{ -1} (x) | = 1 $.  Therefore $\pi $ is trivial.
\end{proof}

This finishes the proof of Proposition \ref {characterization}.

\vspace {2mm} \noindent {\bf Restricting the Markings:} \rm Finally, we show that we can restrict our representations $(\rho_i) $ to a subgroup of $\pi_1 (S, p) $ to create a sequence with the properties desired for our example.  

\begin{claim}
After passing to a subsequence, there exists an index $2 $ subgroup $\Gamma \subset \pi_1 (S, p) $ such that for each $i > 0 $, $$(\sigma_i)_* |_\Gamma: \Gamma \to \pi_1 (N_i,\sigma_i (p)) $$ is a surjection.
\end{claim}
\begin{proof} 
The proof consists of two simple observations.  First, if $M$ is a handlebody then there is an index $2 $ subgroup $\Gamma \subset \pi_1 (\partial M) $ that surjects onto $\pi_1 (M) $.  One can take, for example, the kernel of the map taking an element of $\pi_1 (\partial M) $ to its mod-$2 $ intersection number with any longitude in a standard meridian-longitude basis for $H_1 (\partial M) $.  Therefore, we can construct for each $i $ an index $2 $ subgroup $\Gamma_i $ with $(\sigma_i)_* : \Gamma_i \to \pi_1 (N_i ,\sigma_i (p)) $ a surjection.  Since there are only finitely many index $2 $ subgroups of $\pi_1 (S, p) $, we can pass to a subsequence so that a single $\Gamma \subset \pi_1 (S, p) $ works for all $i $. 
\end{proof}

Consider now the sequence of representations $\rho_i |_\Gamma : \Gamma \to \PSL (2,\BC) $.  The claim implies that $\rho_i (\Gamma) = \rho_i ( \pi_1 (S, p)) $, so $$\rho_i (\Gamma) \to \rho_\infty (\pi_1 (S, p)) $$ geometrically.  Now $\rho_i \to \rho_\infty $ algebraically, so $\rho_i |_\Gamma \to \rho_\infty |_\Gamma $.  Therefore, since $\rho_\infty $ is faithful, $\rho_\infty (\Gamma) $ has index $2 $ in $\rho_\infty (\pi_1 (S, p)) $.  We have therefore provided a sequence of representations converging algebraically to a faithful representation whose image is an index $2$ subgroup of the geometric limit, and so that the geometric limit has no parabolics.  This concludes our example.

\vspace {2mm} \noindent {\bf An Alternate Method:} \rm The difficult part of the example above is constructing a sequence of hyperbolic $3$-manifolds $(N_i) $, each homeomorphic to the interior of a genus $g $ handlebody, that converges geometrically to a hyperbolic $3$-manifold $N_\infty $ homeomorphic to $\Sigma_g \times \BR $.  Such a sequence can be assembled differently by working backwards from the limit.  

Fix a hyperbolic $3$-manifold $N_\infty \cong \Sigma_g \times \BR $ with no cusps, one degenerate end and one convex-cocompact end.  Such manifolds exist, for instance, by \cite [Theorem 3.7] {McMullenrenormalization}.  If we use that construction then a theorem of McMullen \cite {McMullen-cusps} shows that $N_\infty $ is an algebraic limit of one-sided maximal cusps.  (In fact this is true no matter how $N_\infty $ is constructed, but that relies on the recent resolution of the Density Conjecture for surface groups \cite {ELC2}). This means that there is a sequence of marked hyperbolic $3$-manifolds $ M_{i} \cong \Sigma_g \times \BR $ converging algebraically to $N_\infty $, where each $M_i $ has one convex cocompact end and one maximally cusped end.  

As $N_\infty $ has no cusps, Theorem \ref{AC} implies that there are base points $p_i \in M_i $ so that $(M_i, p_i) \to (N_\infty, p_\infty) $ geometrically.  Furthermore, $p_i $ can be chosen to lie on the component of $\partial CC (M_{ i } ) $ facing the convex-cocompact end of $M_i $.  Note that since the thrice-punctured sphere components of $\partial CC (M_{i }) $ disappear in the geometric limit, their distances to $p_i $ grow without bound.

\begin{figure}[t]
\label{gluingfigure}
\includegraphics{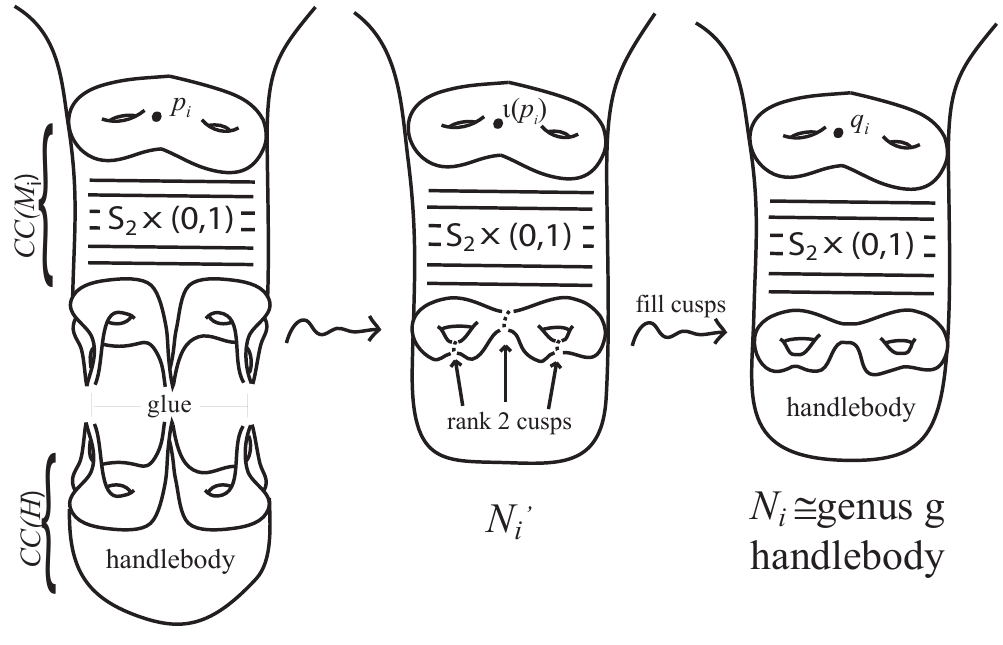}
\caption{Producing long handlebodies by gluing maximal cusps.}
\end{figure}

Passing to a subsequence, we may assume that the pinched pants decompositions on the maximally cusped ends of $M_{i } $ all have the same topological type.  In other words, we may choose a pants decomposition $P \subset \Sigma_g $ and homeomorphisms $$ { M }_{i }\ \partial_c M_i \cong (\Sigma_g \times [0,1]) \setminus (P \times \{0\}) $$ for all $i $.  Pick an identification of $\Sigma_g $ with the boundary of a genus $g $ handlebody $H $ so that $P $ is a pinchable collection of curves on $\partial H $; the latter condition can be ensured, for instance, by composing any fixed identification with a high power of a pseudo-anosov homeomorphism of $\partial H $ whose attracting lamination lies in the Masur domain.  This allows us to endow $H $ with a geometrically finite hyperbolic metric in which $P $ has been pinched.  Then both $\partial CC (H) $ and the bottom boundary components of $CC (M_{ i } ) $ are identified with $\Sigma_g \setminus P $, so we may use Lemma \ref{gluingcusps} to glue them together.  This produces a sequence of hyperbolic $3$-manifolds $N_i' $ equipped with isometric embeddings of $CC (H)$ and the subset $K_i \subset M_{i } $ that is the union of the convex core of $M_{i } $ and its convex-cocompact end; we will denote the inclusion of the latter by $\iota_i : K_i \to N_i' $.  Since the frontier of $K_i $ in $M_{i } $ consists of the thrice-punctured sphere components of $\partial CC (M_{i }) $, its distance to the base point $p_i \in M_{i } $ tends to infinity with $i $.  The same holds for the distances from $\iota_i (p_i) $ to the frontier of $\iota_i ( K_i ) \subset N_i' $.  Therefore, the sequence of based manifolds $(N_i',\iota_i (p_i)) $ converges geometrically to the same limit, $ (N_\infty, p_\infty) $, as our original sequence.  

Observe that $N_i' $ is homeomorphic to the manifold obtained from $H $ by pushing $P \subset \partial H $ into the interior of $H $ and then drilling it out.  The curves in $P $ correspond to rank 2 cusps of $N_i' $, so we may choose $(x,y) $-coordinates for the Dehn filling space of each cusp so that $(1,0) $ corresponds to filling a curve that is contractible in $H $ and $(0,1) $ represents filling a curve homotopic into $P $.  Then the manifold $N_{i, n }' $ obtained from $(1,n) $-Dehn filling  each cusp of $N_i' $ is homeomorphic to $H $ (compare with \cite[Section 3]{KT}).  If $n $ is large, an extension of Thurston's Dehn filling theorem due to to Bonahon-Otal \cite{Bonahon-Otal} and Comar \cite{Comar} implies that $N_{i, n }' $ admits a unique hyperbolic structure with the same conformal boundary as $N_i' $.  Furthermore, there are base points $ p_{i, n } \in N_{i, n }' $ such that $ (N_{i, n }' , p_{i, n })\to (N_i',\iota_i (p_i)) $ geometrically.  An appropriate sequence $(n_i) $ can then be chosen so that $ (N_{i, n_i }', p_{i, n_i }) $ converges geometrically to $ (N_\infty, p_\infty) $.  Setting $N_i = N_{i, n_i }' $ and $q_i = p_{i, n_i} $ finishes our work.\qed

\begin{named}{Example \ref{ex2}}
Let $\Gamma \cong \pi_1 (\Sigma_3) \star \BZ$ be the fundamental group of a compression body with exterior boundary of genus $4$ and connected interior boundary of genus $3$. There is a sequence $(\rho_i)$ in $\calD (\Gamma) $ converging algebraically to $\rho$ and geometrically to $G$ such that:
\begin{itemize}
\item $G$ does not contain any parabolic elements.
\item $\rho(\Gamma)$ has infinite index in $G$.
\end{itemize} 
\end{named}

In the following, $\hat \Gamma \cong \pi_1 (\Sigma_2) \star \BZ $ will be the fundamental group of a compression body with genus $3 $ exterior boundary and genus $2 $ interior boundary.  The group $\Gamma $ will be the subgroup generated by the fundamental group of a double cover of the interior boundary and a loop going around the remaining handle.

The representations here are constructed from those in the previous example by using Klein-Maskit combination:

\begin{named}{The Combination Theorem}[see \cite{Japaner}]
Let $G_1 $ and $G_2$ be two discrete and torsion free subgroups of $\PSL_2\BC $.  Suppose that there exist fundamental domains $D_i \subset \Omega (G_i) $ for $G_i $, each containing the exterior of the other.  Then $G =\left <G_1,G_2\right >$ is discrete, torsion free and is isomorphic to $G_1 \star G_2 $.  Moreover, if the groups $G_i $ do not contain parabolics then the same is true for $G $.
\end{named}

\begin{figure}

\centering
\includegraphics{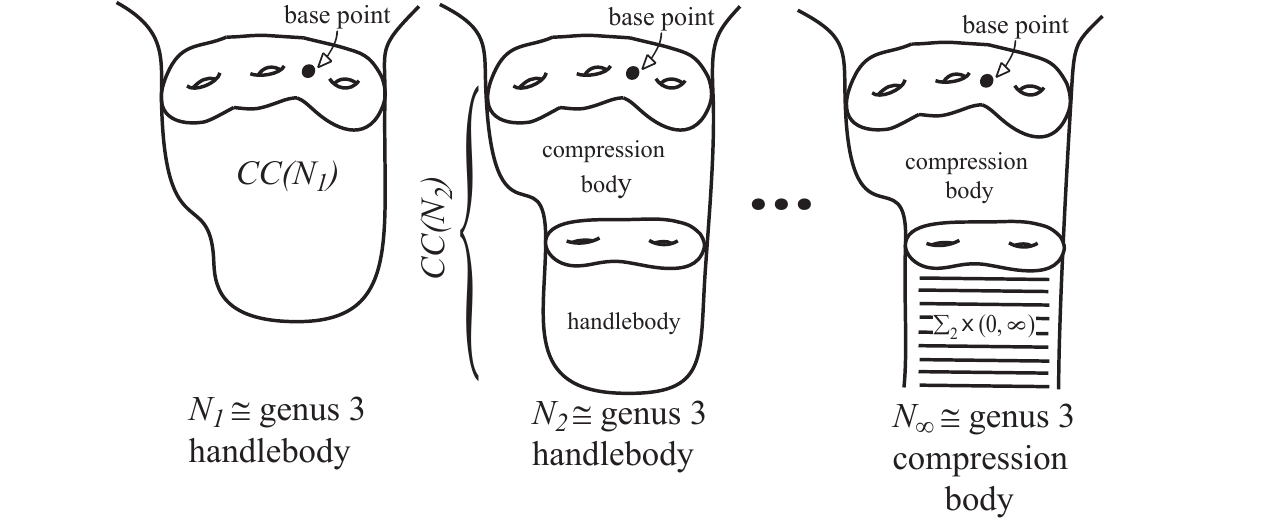}
\caption{The geometric convergence $N_i\to N_\infty$.\label{characterization2}}  
\end{figure}

Let us recall the players from Example \ref {ex1}.  For clarity, we will append apostrophes to all letters that represent objects from the first example.  So, $ (\rho_i') $ is a sequence of representations in $\calD (\pi_1 (\Sigma_2)) $ converging strongly to a faithful representation $\rho_\infty' $ without parabolics, and $\Gamma' \subset \pi_1 (\Sigma_2) $ is an index $2 $ subgroup with $\rho_i (\Gamma') = \rho_i (\pi_1 (\Sigma_2) )$ for all $i $. For convenience, set $G_i '= \rho_i (\pi_1 (\Sigma_2)) $ and $G_\infty' = \rho_\infty (\pi_1 (\Sigma_2)) $.

Our previous analysis implies that the convex cores of $\Hyp^3 / G_i' $ converge to the convex core of $\Hyp^3 / G_\infty '$ in the Gromov-Hausdorff topology, so the limit sets $\Lambda (G_i') $ converge to $\Lambda (G_\infty') $ in the Hausdorff topology.  Since the domains of discontinuity $\Omega (G_i') $ and $\Omega (G_\infty')$ are all connected and nonempty, it follows that there are fundamental domains $D_i $ for $G_i ' \actson \Omega (G_i') $ converging in the Hausdorff topology to a fundamental domain $D_\infty $ for $G_\infty' \actson \Omega (G_\infty') $. 

 Pick some loxodromic element $\alpha \in \PSL (2,\BC) $ with fixed points contained in the interior of $D_\infty $ (which is connected).  Moreover, assume that its translation distance is large enough so that there is a fundamental set for $\Omega (\left<\alpha \right>) $ whose complement is entirely contained within $D_\infty $.  After discarding a finite number of terms, its complement will also be contained in $D_i $ for all $i $. We now construct new representations $$ \rho_i: \hat\Gamma \longrightarrow \PSL (2,\BC), \text { where } \hat \Gamma :=\pi_1 (\Sigma_2) \star \BZ $$ from $\rho_i '$ by sending $1 \in \BZ $ to $\alpha $..If $G_i \subset \PSL (2,\BC) $ is the image of $\rho_i $, the Klein-Maskit combination theorem implies that $G_i = G_i '\star \left <\alpha \right >$ and is discrete, torsion free and has no parabolics.     The same statements apply when $i = \infty $, showing that that $\rho_\infty$ is faithful with image $$G_\infty = G_\infty' \star \left <\alpha\right >. $$
The manifolds $N_i= \Hyp^ 3 / G_i , \ i = 1, 2,\ldots $ are all convex-cocompact and homeomorphic to the interior of a genus $3 $ handlebody.  The manifold $ N_\infty  = \Hyp^3 / G_\infty $ is homeomorphic to a compression body with genus $3 $ exterior boundary and connected, genus $2 $ interior boundary.

 It follows, for instance, from the argument in \cite[Prop 10.2]{AC1} that $\rho_i \to \rho_\infty $ strongly.
  Consider the subgroup $$\Gamma := \Gamma' \star \BZ \  \ \subset \ \  \pi_1 (\Sigma_2) \star \BZ=: \hat \Gamma. $$ Then $\rho_i |_{\Gamma} $ converges algebraically to $\rho_\infty |_{\Gamma } $. Since $\rho_i' (\Gamma') = \rho_i' (\pi_1 (\Sigma_2)) $ for all $i \in \BN $,  we also have that
$$\rho_i (\Gamma) = \rho_i (\hat \Gamma) \text { for all } i \in \BN.$$ So, $\rho_i (\Gamma)$ converges geometrically to $\rho_\infty (\hat\Gamma) =G_\infty $.  Since $\rho_\infty $ is faithful and $[\hat\Gamma : \Gamma] =\infty $, it follows that $\rho_\infty (\Gamma) $ has infinite index in $\rho_\infty (\hat\Gamma) =G_\infty $.  Therefore $\rho_i : \Gamma \to \PSL (2,\BC) $ is a sequence of representations for which the algebraic limit has infinite index in the geometric limit.  As mentioned above, the geometric limit $G_\infty$ has no parabolics.  Since $\BH^ 3 / \rho_\infty (\Gamma)$ is homeomorphic to a compression body with exterior boundary of genus $4 $ and connected interior boundary of genus $3 $, we have provided the desired example.
\vspace{2mm}

\begin{named}{Example \ref{ex3}}
Let $\Gamma$ be the fundamental group of a compression body with exterior boundary of genus $4$ and connected interior boundary of genus $3$.  There is a sequence $(\rho_i)$ in $\calD (\Gamma) $  converging algebraically to a representation $\rho$ and geometrically to a group $G$ such that:
\begin{itemize}
\item $\rho(\Gamma)$ does not contain any parabolic elements.
\item $G$ is infinitely generated.
\end{itemize} 
\end{named}

In \cite{Thu86B}, Thurston exhibited a sequence of representations of a closed surface group into $\PSL (2,\BC) $ converging geometrically to a group that is not finitely generated.  However, the algebraic limit of these representations contains parabolic elements.  The idea here is to attach pieces of the handlebodies in Example \ref{ex2} to the manifolds in his sequence so that the parabolics are hidden outside the algebraic limit. 

To facilitate such a combination, we must build a variant of Example \ref{ex2} in which each of the handlebodies in the sequence is maximally cusped.  Let $M$ be a compression body with genus $3 $ exterior boundary $S_E $ and connected, genus $2 $ interior boundary $S_I $.  Assume that $P_E \subset S_E $ is a pants decomposition consisting of curves in the Masur domain of $M $.

\begin{claim}  There is a sequence of maximally cusped pointed hyperbolic $3$-manifolds $ ( N_i ) $, each homeomorphic to the interior of a genus $3 $ handlebody, that converges geometrically to a hyperbolic $3$-manifold $ N_\infty $ homeomorphic to the interior of $M $ in which $P_E $ has been pinched.  Moreover, the subsets $ CC (N_i) \subset N_i $ converge to $CC (N_\infty) $ geometrically.  
\end{claim}  

\begin{proof} The sequence $(N_i) $ will be constructed in two steps.  First, we will produce a sequence of hyperbolic $3$-manifolds $ M_i $ homeomorphic to the interior of $M$ in which both ends are maximally cusped.  The sequence will converge geometrically to the manifold $ N_\infty $ referenced in the statement of the claim.  We will then use the same gluing trick exploited in Example \ref{ex1} to cap off the interior ends of each $M_i $ without changing the sequence's geometric limit, thus producing the desired sequence of handlebodies $ (N_i) $.  

Fix pants decompositions $P_E $ and $ P_I $ for the boundary components of $M$, and assume that every curve in $ P_E $ lies in the Masur domain of $M $.  If we choose a pseudo-anosov diffeomorphism $f : S_I \to S_I $, then for each $i $ we have a new pants decompositions $f^i (P_I) $ for $S_I $.
It is not hard to check that for each $i $, $P_E \cup f^i ( P_I) $ is a pinchable collection of curves on $\partial M $ (see Section \ref{maximalcusps}).  So, there is a sequence of marked hyperbolic $3$-manifolds $M_i \in AH (M) $ in which $P_E \cup f^i (P_I) $ have been pinched.  Note that in fact $M_i $ lies in the deformation space $AH(M, P_E) $ of hyperbolic structures on the interior of $M$ in which the curves of $P_E $ represent parabolics.  Since $P_E $ consists of curves lying in the Masur domain, the pared manifold $(M, P_E) $ has incompressible and acylindrical boundary.  It follows from a theorem of Thurston, \cite[Theorem 7.1]{Thu86A}, that $AH (M,P_E) $ is compact.  So after passing to a subsequence, we may assume that $ (M_i) $ converges algebraically to some $ N_\infty \in AH (M, P_E) $.

We claim that the only parabolic loops in $ N_\infty $ are those that are freely homotopic into $P_E $.  The end of $ N_\infty $ facing $S_E $ is maximally cusped by $P_E $, so there is no room for additional cusps there.  It therefore suffices to show that the end facing $S_I $ has no cusps.  Work of Thurston, Bonahon and Brock, implies that there is a continuous map 
$$\length : AH (M) \times \CM \CL (S_I)\to\BR $$ 
that extends the function that assigns to an element $ N \in AH (M) $ and a simple closed curve $\gamma \in S_I $ the shortest length of a curve in $ N $ homotopic to $\gamma $ (see \cite{length}). Since $\length_{M_i } (f^i (P_I)) = 0 $ for all $i $, in the limit we have $\length_{M_\infty } (\lambda) = 0 $, where $\lambda $ is the attracting lamination of $f $.  This implies that $\lambda $ cannot be geodesically realized by a pleated surface in $ N_\infty $ homotopic to $S_I $.  Let $\hat { N }_\infty $ be the cover of $ N_\infty $ corresponding to $\pi_1 (S_I ) $.  The end of $ N_\infty $ facing $S_I $ lifts homeomorphically to an end $\CE $ of $\hat {N}_\infty $.  The other end of $\hat {N}_\infty $ has no cusps and is therefore convex-cocompact by the Tameness Theorem and Thurston's covering theorem.  Since $\lambda $ is filling and unrealizable in $\hat {N}_\infty $, the argument in \cite[Theorem 6.34]{Japaner} then shows that $\CE $ is degenerate with ending lamination $\lambda $.  In particular, $\CE $ has no cusps.  Projecting down, the same is true for the end of $ N_\infty $ facing $S_I $.
 
This shows that the convergence $M_i \to N_\infty $ is type preserving, so Theorem \ref{AC} implies that the convergence is strong.  So, base frames for $M_i $ can be chosen so that the sequence converges geometrically to $ N_\infty $.  The rest of the argument follows that given at the end of Example \ref{ex1}.  Fix a genus $2 $ handlebody $H $ and choose a hyperbolic metric on its interior in which a pants decomposition $P \subset \partial H $ with the same topological type as $P_I \subset S_I $ has been pinched.  We can then create for each $i $ a hyperbolic $3$-manifold $N_i' $ by removing from $M_i $ the component of $M_i \setminus CC (M_i) $ facing $S_I $ and gluing $CC (H ) $ in its place.  As in Example \ref{ex1}, the sequence $(N_i' ) $ converges geometrically to $ N_\infty $.  Performing an appropriate Dehn filling on each $ N_i' $ yields a sequence of hyperbolic $3$-manifolds $ (N_i) $, each homeomorphic to the interior of a genus $3 $ handlebody, that also converges geometrically to $N_\infty $.  

We must show that $CC (N_i) \to CC (N_\infty ) $ geometrically.  First, every component of $\partial CC (N_\infty) $ is contained in a geometric limit of some sequence of components of $\partial CC (N_i) $.  These are all thrice-punctured spheres, however, so in fact we have that for large $i $ there are components of $\partial CC (N_i ) $ that closely approximate each component of $\partial CC (N_\infty) $.  However, $ \partial CC (N_i) $ and $ \partial CC (N_\infty) $ both have $4 $ components, so for large $i $ they must almost coincide.  From this, it is easy to check that $CC (N_i) \to CC (N_\infty) $ geometrically. \end{proof}

Fix a geometrically finite hyperbolic structure on $\Sigma_3 \times \BR $ in which both ends are maximally cusped, and let $C $ be its convex core.  Then there are pants decompositions $P_{ +, - }  \subset\Sigma_3$ so that $$C \cong (\Sigma_3\times [- 1, + 1]) \setminus (P_- \times \{- 1\} \cup P_+ \times \{+1\}) , $$ and we label the components of $\partial  C $ as positive and negative accordingly, so that $\partial C = \partial_+ C \cup \partial_- C $.  Let $E $ be the union of all components of the complement of $C $ that face its positive boundary components.  Assume that the pairs $(\Sigma_3, P_{ +, -}) $ and $(S_E, P_E) $ have the same topological type, and that every curve in $P_{ + } $ intersects some curve in $ P_- $.  

We now glue these pieces to the manifolds $N_i $ and $N_\infty $ from the previous Lemma.  To begin with, let $$N_\infty' = CC (N_\infty) \sqcup_h C \sqcup_g C \sqcup_g \cdots.  $$   Here, the gluing maps $h : \partial CC (N_\infty) \to \partial_- C $ and $g : \partial_- C \to \partial_+ C $ can be any isometries that extend to maps $ (S_E, P_E) \to (\Sigma_3, P_-) $ and $(\Sigma_3, P_+) \to (\Sigma_3, P_-) $.  This ensures that $N_\infty' $ is constructed as in Lemma \ref{gluingcusps}.  Note that the inclusion map $CC (N_\infty) \to N_\infty' $ is $\pi_1 $-injective.

 Next, for large $i $ geometric convergence and the gluing map $h $ determine an identification $h_i : \partial CC (N_i) \to \partial_- C $.  Define $$N_i' = CC (N_i) \sqcup_{ h_i } \underbrace{ C \sqcup_g \cdots \sqcup_g C }_{i \text { times } } \sqcup_{id } E.  $$
Recall that $ CC (N_i) \to CC (N_\infty) $ geometrically.  From this it follows that if $ N_i' $ is given the base frame of $N_i $ produced in the previous Lemma, then $(N_i') $ converges geometrically to $N_\infty' $.  

As in Example \ref{ex1}, performing $(1,n) $-Dehn filling on each of the cusps in $N_i' $ produces, for large $n $, a new hyperbolic manifold $ N_{i, n }' $ homeomorphic to the interior of a genus $3 $ handlebody.  Also, an appropriate diagonal sequence $ N_i^*= N_{i, n_i }' $ can be chosen to converge geometrically to $N_\infty' $.  In summary, we have proven the following claim.

\begin{figure}
\includegraphics{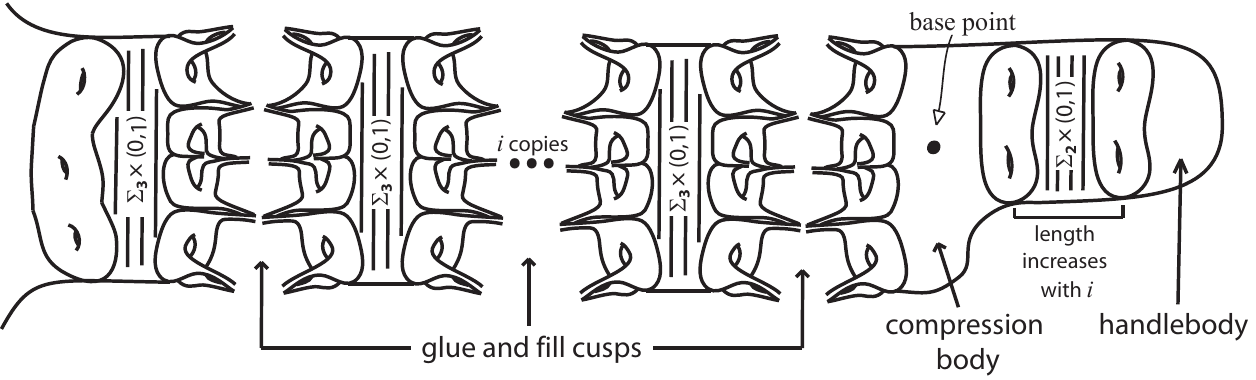}
\caption{The construction of the manifold $N_i^* $.}
\label{characterization3}
\end{figure}

\begin{claim}  There is a sequence of convex-cocompact pointed hyperbolic $3$-manifolds $N_i^* $, each homeomorphic to the interior of a genus $3 $ handlebody, that converges geometrically to a hyperbolic $3$-manifold $N_\infty' $ with infinitely generated fundamental group.
\end{claim}

We now obtain the sequence of representations advertised in the statement of this example by marking the manifolds $N_i^* $ appropriately.  Recall that the fundamental group of the compression body $M$ splits as a free product $$\pi_1 (M) = \pi_1 (S_I) \star \left < \alpha \right > \cong \pi_1 (\Sigma_2) \star \BZ, $$  for some element $\alpha \in \pi_1(M) $.  The inclusion map $M \cong CC (N_\infty) \to N_\infty' $ is $\pi_1 $-injective, so it determines an embedding $$\rho_\infty : \pi_1 (\Sigma_2 ) \star \BZ \longrightarrow \pi_1 (N_\infty') . $$  Then $\rho_\infty $ identifies a finite generating set for $\pi_1 (\Sigma_2) \star \BZ $ with a finite set of loops in $ N_\infty' $.  For large $i $, geometric convergence provides an almost isometric embedding of these loops into $N_i^* $; therefore, there are induced homomorphisms $$\rho_i : \pi_1 (\Sigma_2) \star \BZ \longrightarrow \pi_1 (N_i^*).  $$  In fact, $\rho_i $ is surjective, and therefore is a marking of $\pi_1 (N_i^*) $.  

The sequence of marked hyperbolic manifolds $ (N_i^*,\rho_i) $ converges algebraically to the cover of $N_\infty' $ corresponding to the image of $\rho_\infty $, and converges geometrically to $N_\infty' $ (as noted above).  Note that $\pi_1 (N_\infty' ) $ is not finitely generated.  We are not quite done, however, because the algebraic limit here has cusps.  To hide the cusps, we will use the finite index trick exploited in Example \ref{ex1}, but for this to work the pants decomposition $ P_E $ used above must be chosen more carefully.

\begin{claim}
There is a pants decomposition $P_E \subset S_E $ consisting of curves $\{\gamma_1,\ldots,\gamma_n \} $ in the Masur domain with the property that no power of any $\gamma_i $ is conjugate into any subgroup of $\pi_1(M) $ of the form $\Gamma' \star \left < \alpha \right > $, where $\Gamma' < \pi_1 (S_I, p) $ has index $2 $.  
\end{claim}

Deferring the proof for a moment, pick as in Examples \ref{ex1} and \ref{ex2} an index $2 $ subgroup $\Gamma' \subset \pi_1 (\Sigma_2) $ so that the restriction of $$\rho_i: \pi_1 (\Sigma_2) \star \BZ \longrightarrow \PSL (2,\BC) $$ to the subgroup $\Gamma = \Gamma' \star \BZ $ surjects onto each $\pi_1 (N_i^*) $ .  If in constructing $ (N_i^*) $, the pants decomposition $P_E $ is chosen as indicated in the above claim, there will be no parabolics in the algebraic limit of $ (\rho_i |_{\Gamma }) $.  However, since this is still a sequence of markings for $ N_i^* $, the geometric limit will be $N_\infty' $, which has infinitely generated fundamental group.  This finishes the example.

\begin{proof}[Proof of Claim]
Although the claim is purely topological, the proof we give uses $3 $-dimensional hyperbolic geometry.  It would be nice to give a more straightforward proof; also, it is possible that the second part of the conclusion is satisfied by any pants decomposition of curves in the Masur domain.

To begin, construct by some means a hyperbolic manifold $N $ homeomorphic to the interior of $M$ that has no cusps and in which both ends are degenerate.  One way to produce $N $ is as follows.  First, find a hyperbolic manifold homeomorphic to $\Sigma_3 \times \BR $ that has one degenerate end and one maximally cusped end.  This can be done using an argument similar to the construction of $N_\infty $ above.  We can then glue its convex core to the convex core of $ N_\infty $ as in Lemma \ref{gluingcusps} and fill the resulting cusps to create a totally degenerate hyperbolic manifold $N $ homeomorphic to the interior of $M$.  See \cite {Bonahon-Otal} for information on Dehn filling geometrically infinite manifolds.

Fix an index $2 $ subgroup $\Gamma' \subset \pi_1 (S_I, p) $, and let $N_{\Gamma'} $ be the cover of $N $ corresponding to $\Gamma' \star \left < \alpha \right >$.  Then $N_{\Gamma'} $ has a degenerate end homeomorphic to $\Sigma_3 \times [0, \infty) $ that double covers the genus $2 $ end of $N $.  Adjoining a loop in $N_{\Gamma'}  $ representing $\alpha $ to a level surface of this end and thickening produces a compact core for $N_{\Gamma'} $ homeomorphic to the interior of a compression body with genus $4 $ exterior boundary and connected, genus $3 $ interior boundary.  The tameness theorem of Agol \cite{Agol} and Calegari-Gabai \cite{Calegari-Gabai} implies that $N_{\Gamma'} $ is itself homeomorphic to the interior of such a compression body, and a theorem of Canary \cite{Canary-ends} implies that its genus $4 $ end, $\hat {\CE } $, is either degenerate or convex-cocompact.  It cannot be that both ends of $N_{\Gamma'} $ are degenerate, for then Canary's covering theorem (Theorem \ref {thecoveringtheorem}) would imply that $\Gamma' \star \left < \alpha \right >$ is finite index in $\pi_1 (N) $.  Therefore, $\hat {\CE } $ is convex-cocompact.  

Since the genus $3 $ end, $\CE $, of $N $ is degenerate, it has an ending lamination $\lambda \subset S_E $. Canary has shown that $\lambda $ lies in the Masur domain of $M $ (see \cite{Canary-ends} for a proof of this and the uniqueness of $\lambda $).  Choose a pants decomposition $P_E \subset S_E $ consisting of curves that lie close to $\lambda $ in $\CP \CM \CL (S_E) $.  The Masur domain of $M$ is an open subset of $\CP \CM \CL (S_E) $, so we may assume that each curve in $P_E $ lies inside of it.  Furthermore, their geodesic representatives in $N $ lie very deep inside of $\CE $.  As $\hat {\CE } $ is convex-cocompact, the convex core of $N_{\Gamma'} $ covers a subset of $N $ that has bounded intersection with $\CE $.  So, we may assume that the geodesic representatives of curves in $P_E $ do not intersect its image.  If some power of a curve in $P_E $ were conjugate into $\Gamma' \star \left < \alpha \right>$, then its geodesic representative in $N $ would lift to a closed geodesic in $N_{\Gamma'} $.  Every closed geodesic in $N_{\Gamma'} $ is contained in its convex core, so this is impossible.

This last paragraph shows that as long as the curves in $P_E $ are chosen within a small neighborhood around the ending lamination $\lambda $, then no powers of them are conjugate into subgroup of $\pi_1 (M) $ of the form $\Gamma' \star \left < \alpha \right> $, where $\Gamma' $ is a \it fixed \rm index 2 subgroup of $\pi_1 (S_I) $.    However, there are only finitely many index 2 subgroups of $\pi_1 (S_I) $, so it follows that if we choose $P_E $ from the intersection of all such neighborhoods we can ensure that no power of any of its curves is conjugate into \it any \rm such $\Gamma' \star \left < \alpha \right> $.
\end{proof}

\section{Proof of Theorem \ref{max-cyclic}}\label{sec:max-cyclic}

Before beginning the bulk of the proof, we will present a technical lemma whose proof requires a bit of differential geometry.  Afterwards, Theorem \ref{max-cyclic} will follow from purely topological arguments.

 Recall from Proposition \ref{exhaustion} that a hyperbolic $3$-manifold $M $ with finitely generated fundamental group and no cusps contains a compact core $C \subset M $ for which each component of $\partial C $ facing a convex-cocompact end of $M $ is smooth and strictly convex.  For convenience, let $S_{ cc } $ be the union of those components of $\partial C $ facing convex cocompact ends and $E_{ cc } $ the union of the adjacent components of $M \setminus C $.  Then $E_{ cc } $ is homeomorphic to $ S_{ cc } \times (0,\infty) $ via `radial coordinates': $$ R: S_{ cc } \times (0,\infty) \to E_{ cc }, \, \, \, R(x, t) =\exp_x (t \nu (x)), $$ where $\nu $ is the outer unit-normal vector field along $ S_{ cc } $.  

If $f : C \to N $ is a smooth immersion into some complete hyperbolic $3$-manifold $ N $, it has a natural \emph{radial extension} $\bar f : C \cup E_{cc } \to N $.  Namely, there is a radial coordinate map along the image of $ S_{ cc } $:   $$R_f : S_{ cc } \times (0,\infty) \to N, \, \, \, R_f (x , t) = \exp_{f(x)} (t \nu_f (x)), $$  where $\nu_f (x) $ is the unit vector in $TN_{ f (x) } $ orthogonal to $df_x (TS_x) $ that points away from $f (C) $, and one can then define $$\bar { f } (p) = \begin{cases} R_f \circ R^{-1 } (p) & p \in E_{cc } \\ f (p) & p \in C.  \end{cases} $$
Observe that $\bar f $ is continuous, and differentiable everywhere but on $S_{cc } $.  

In the situations where we will find radial extensions useful, the map $f $ will be very close in the $C^ 2 $ topology to a \it Riemannian \rm immersion.  In particular, the (strict) convexity of the surface $S_{cc } $ will persist in the image.  This implies a convenient regularity in the radial extension:

\begin{lem}
\label{Lemmabilipschitz}
If $f : C \to N $ is a smooth immersion with $ f (S_{cc }) $ convex, there exists some $L>0 $ so that for all $p \in E_{ cc } $ and $v \in T M_p $, $$\frac {1} {L} \leq \frac { | |d\bar{ f }_p (v) | | } { | | v | | } \leq L.  $$
\end{lem}

% Here, $f |_{ S_{cc } } $ is \it convex \rm if $\left <\frac\DD{ds}\nu_f (g (s)), (f \circ g)' (s)\right>\geq0 $, where $g : \BR \to S_{cc } $ is any smooth curve and $ \DD $ is the Levi-Civita connection for $N $.

The following global statement comes from applying Lemma \ref{Lemmabilipschitz} and a compactness argument on $C $.

\begin{kor}[Radial Extensions are Locally Bilipschitz]
\label{local-bilipschitz}
If $f : C \to N $ is a smooth immersion with $ f ( S_{cc } ) $ convex, there exists some $L>0 $ so that every $p \in C \cup E_{cc } $ has a neighborhood on which $f $ is $L $-bilipschitz.
\end{kor}

\begin{proof}[Proof of Lemma \ref{Lemmabilipschitz}]
Consider a component $E \subset E_{ cc } $, and let $S \subset S_{ cc} $ be the adjacent component of $\partial C $.  Here, $\bar f $ is the composition $R_f \circ (R)^{-1} $ of radial coordinate maps, so to prove the Lemma it suffices to find a constant $L $ so that for all  $ (x, t)\in S \times (0,\infty) $ and $v\in TS_x \times \BR $,\begin{equation}\label{trouble}\frac {1} {L} \leq \frac { | | (dR_f)_{ (x, t) } (v) | | } { | |dR_{ (x, t) } (v) | | } \leq L . \end{equation} Since the ratio is one when $v $ is contained in the $\BR $ factor, from now on we will assume $v \in TS_x $.

We first estimate $\Vert dR_{(x,t)}(v)\Vert$. Given $x\in S $ and $v \in TS_x $, let $g(s)$ be a curve in $S$ with $g(0)=x$ and $g'(0)= v $, and consider the geodesic variation $$\gamma_s (t) : =R (g(s),t)=\exp_{g(s)}(t\nu (g(s))) $$  The corresponding Jacobi field $J_{ x, v } (t)=\frac d{ds}R (g(s),t) |_{ s =0} $ along the geodesic $\gamma_0(t) $ then satisfies $J_{ x, v }(t)=dR_{(x,t)}(v)$.

The Jacobi-field $J_{ x, v }(t)$ is determined by its initial conditions
\begin{align*} 
J_{ x, v }(0) & =dR_{(x,0)} (v), \, \, \text{and} \\
\frac\DD{dt}J_{ x, v }(t) |_{ t =0} &=\frac\DD{dt}\frac d{ds}\exp_{g(s)}\big (t\nu (g (s))\big) |_{t =0} \\& =\frac\DD{ds}\frac d{dt}\exp_{g(s)}\big (t\nu (g (s))\big) |_{ t =0}\\
&=\frac\DD{ds}\nu (g (s)) = \DD_{dR_{ (x,0) } ( v) } \nu \end{align*}
Therefore, we have $$J_{ x, v } (t) = \cosh(t)E_1(t)+\sinh(t)E_2(t ), $$ where $E_1(t)$ and $E_2(t)$ are the parallel vector fields along $\gamma_0 $ with $E_1(0)=dR_{ (x, 0) } (v)  $ and $E_2(0)= \DD_{dR_{ (x,0) } ( v) } \nu $.  That the right-hand side satisfies the Jacobi equation follows quickly from the fact that $E_1(t) $ and $E_2(t) $ are both orthogonal to $\gamma_0' (t) $. 

The triangle inequality, together with the fact that the vector fields $E_1$ and $E_2$ have constant length, shows that
$$\Vert J_{ x, v }(t)\Vert\le\cosh(t)(\Vert dR _{(x,0)} (v)\Vert+\Vert\DD_{dR_{ (x,0) } ( v) }\nu \Vert).  $$
On the other hand we have 
\begin{align*}
\Vert J_{ x, v }(t)\Vert^2=&\cosh(t)^2\Vert dR_{(x,0)} (v)\Vert^2+\sinh(t)^2\Vert\DD_{dR_{ (x,0) } ( v) } \nu \Vert^2\\
&+2\cosh(t)\sinh(t)\langle dR _{(x,0)}v,\DD_{dR_{(x,0)} (v)}\nu\rangle
\end{align*}
By convexity of $S_{cc } $, the last term in the sum is positive. A little bit of algebra and the fact that $J_{ x, v }(t)=dR_{(x,t)}(v)$ yields
\begin{equation}\label{bound1}
\frac{\sinh(t)}2\le \frac{\Vert dR_{(x,t)}(v)\Vert}{\Vert dR_{(x,0)} (v)\Vert+\Vert\DD_{dR_{(x,0)} (v)}\nu\Vert}\le\cosh(t)
\end{equation}
Since $ f (S_{cc }) $ is convex, a similar computation shows 
\begin{equation}\label{bound2}
\frac{\sinh(t)}2\le \frac{\Vert d (R_f)_{(f (x),t)}(v)\Vert}{\Vert d (R_f)_{(f (x),t)} (v)\Vert+\Vert\DD_{d (R_f)_{(f (x),t)}v}\nu_f\Vert}\le\cosh(t)
\end{equation}
By compactness the ratio between the denominators in \eqref{bound1} and \eqref{bound2} is uniformly bounded from above and below. In other words, there is some positive constant $c$ with
\begin{equation}\label{bound3}
\frac{\sinh(t)}{2 c \cosh(t)}\le \frac{\Vert d (R_f) _{(f (x),t)}(v)\Vert}{\Vert dR_{(x,t)}(v)\Vert}\le\frac{2 c\cosh(t)}{\sinh(t)}
\end{equation}
for all $(x,t)\in S\times(0,\infty)$ and $v\in T_xS$. If $t $ is constrained away from zero then the lower and upper bounds in \eqref{bound3} are bounded by positive numbers from below and above respectively.  When $t =0 $ both $dR_{ (x, t) } $ and $d (R_f)_{ (x, t) } $ have maximal rank, so constant positive bounds arise from a compactness argument.  This yields \eqref{trouble} and concludes the proof of the Lemma.
\end{proof}

% Moreover, since $\bar { f } $ has maximal rank on $S\times\{0\}$ we need only show a that such a constant $L$ exists for all $(x,t)\in S\times[0,\infty)$ with $t>\epsilon$ for some small positive $\epsilon$.

We are now ready to prove the main result of this section.

\begin{named}{Theorem \ref{max-cyclic}}
Let $\Gamma$ be a finitely generated group and $(\rho_i)$ a sequence in $\calD (\Gamma) $. Assume that $(\rho_i) $ converges algebraically to a representation $\rho \in \calD (\Gamma) $ and geometrically to a subgroup $G$ of $\PSL_2\BC$. If 
\begin{itemize}
\item $\rho(\Gamma)$ does not contain parabolic elements, 
\item maximal cyclic subgroups of $\rho(\Gamma)$ are maximal cyclic in $G$,
\end{itemize}
then $G=\rho(\Gamma)$.
\end{named}

Before going further, set $M_A=\BH^3/\rho(\Gamma) $, $M_G=\BH^3/G$ and $M_i=\BH^3/\rho_i(\Gamma) $ for $i=1,2,\dots$, choose a base frame $\omega_{\BH^3}$ for hyperbolic space $\BH^3$ and let $\omega_i$, $\omega_A$ and $\omega_G$ be the corresponding base frames of $M_i$, $M_A$ and $M_G$ respectively. By  Proposition \ref{convergence-convergence-geom}, the pointed manifolds $(M_i,\omega_i)$ converge geometrically to $(M_G,\omega_G)$.  We may assume that $M_G $ is noncompact, for otherwise Mostow's Rigidity Theorem implies that the sequence $(M_i) $ must be eventually stable, so certainly $G =\rho (\Gamma) $.

We claim that $M_G $ has infinite volume. If not, it has finite volume and $(M_i ) $ is obtained by performing Dehn filling on $M_G $ with larger and larger coefficients \cite[Theorem E.2.4]{Benedetti-Petronio}.  Then there must be parabolics in the algebraic limit $\rho (\Gamma) $.  For otherwise, combining  a result of Canary \cite {Canary-ends} with the Tameness Theorem of Agol \cite{Agol}  and Calegari-Gabai \cite {Calegari-Gabai} implies that $M_A =\BH ^3 / \rho (\Gamma) $ is either convex-cocompact, or has  a degenerate end.  The first case is impossible by Marden's Stability Theorem \cite {Marden74} and the second violates Canary's Covering Theorem (Theorem \ref{thecoveringtheorem}).

There is a covering map $\pi : M_A \to M_G $ induced by the inclusion $\rho (\Gamma) \subset G $, and our goal is to show that this is a homeomorphism.  Recall that Proposition \ref{exhaustion} provides an exhaustion of $M_A $ by compact cores $C \subset M_A $ such that
\begin{itemize}
\item[(1)] if a component $S$ of $\D C$ faces a convex cocompact end of $M_A$ then $S$ is smooth and strictly convex,
\item[(2)] if a component $S$ of $\D C$ faces a degenerate end of $M_A$ then the restriction $\pi\vert_S:S\to\pi(S)$ is a finite covering onto an embedded surface in $ M_G$.
\end{itemize}
Fixing a compact core $C_0 \subset M_A $, we may also assume that all $C $ are large enough to contain $C_0 $ and satisfy the following property:
\begin{itemize}
\item[(3)] $\pi (C_0) \cap \pi (S) = \emptyset $ for any component $S \subset \partial C $ facing a degenerate end of $M_A $.
 \end{itemize}
Then to prove that $\pi $ is a homeomorphism, it clearly suffices to show that $\pi |_C $ is injective for all such compact cores $C \subset M_A $.

Fix a compact core $C \subset M_A $ as described above.  The geometric convergence $(M_i,\omega_i) \to (M_G,\omega_G) $ supplies for sufficiently large $i $ an almost isometric embedding $\phi_i : \pi (C) \hookrightarrow M_i, $ so for large $i $, we have a map $$f_i : C \to M_i , \, \, f_i = \phi_i \circ (\pi|_C)$$ that behaves much like the restriction of a nearly Riemannian covering map.  In fact, $f_i $ is $\pi_1 $-surjective.  For if $S \subset \Gamma $ is a finite generating set then $C $ contains loops based at $\omega_A $ representing the elements of $\rho (S) $; $f_i $ then maps these loops to loops in $M_i $ representing $\rho_i (S) $, which generate $\pi_1 (M_i) \cong \rho_i (\Gamma) $.  We aim to show that $f_i $ is actually an embedding, as the same will then be true for $\pi |_C $.  

We first consider the case where $M_A $ is convex-cocompact, as the proof is particularly simple.  In this case, every component of $\partial C $ is strictly convex and faces a convex-cocompact end of $M_A $, so $f_i $ radially extends (as in the beginning of the section) to a globally defined map $$\bar { f_i } : M_A \to M_i.  $$  We claim that this is a covering map for $i >>0$.  To see this, note that when $i $ is large $f_i $ is $C^2 $-close to a local isometry, so the strict convexity of $\partial C $ persists after applying $f_i $.  Therefore Corollary \ref{local-bilipschitz} applies to show that $\bar { f_i } $ is (uniformly) locally bilipschitz. It is well-known that any locally isometric map between complete Riemannian manifolds is a covering map, and in fact the same proof applies to locally bilipschitz maps.  So, $\bar { f_i } : M_A\to M_i $ is a covering map.  However, $f_i $ is $\pi_1 $-surjective, so its extension $\bar { f_i } $ is a $\pi_1 $-surjective covering map, and therefore a homeomorphism.  This shows that $f_i $ is injective, and in particular $\pi |_C $ is as well.

\begin{figure} 
\centering
\includegraphics{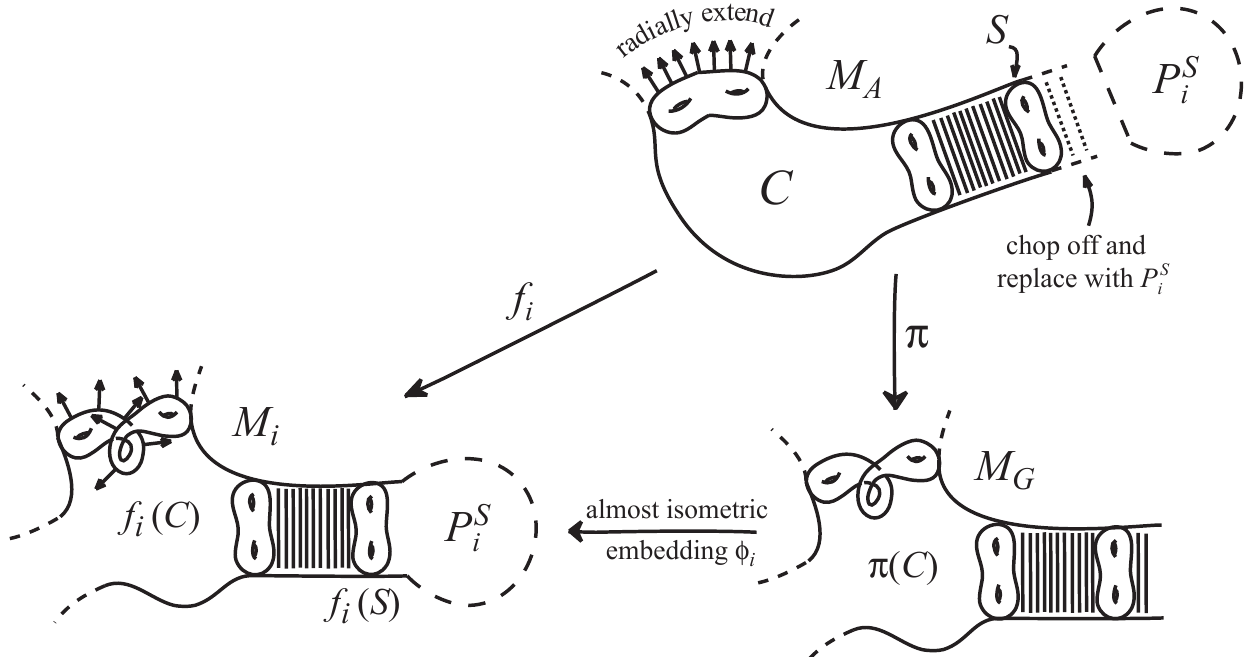}
\caption{Extending $f_i = \phi_i \circ \pi $ to a cover of $M_i $. \label{thm1}}  
\end{figure}

In the general case, the argument needs modification because one cannot radially extend $f_i $ into the degenerate ends of $M_A $.  To deal with this, we will alter the problematic parts of $M_A $ so that an extension of $f_i $ is obvious.

\begin{claim} If $S \subset \partial C $ faces a degenerate end of $M_A $, then the restriction of $f_i $ to $S $ is an embedding with image a separating surface in $M_i $.
\end{claim} 
\begin{proof}
By Condition  $(2) $ above, $\pi |_{ S  } $ is a finite covering onto its image.  The assumption that maximally cyclic subgroups of $\rho (\Gamma) $ are maximally cyclic in $G $ then implies that $\pi |_{ S  } $ is an embedding, for otherwise there would be a loop in $\pi (S ) $ that does not lift to $M_A $, but that has a power which does.

Next, property $(3) $ above implies that every component $S \subset f_i (S ) $ is disjoint from $ f_i (C_0) $.  However, the argument given to show that $f_i | _C $ is $\pi_1 $-surjective also applies to $f_i | _{C_0 }$, so every loop in $M_i $ is homotopic into $f_i (C_0) $ and therefore has trivial algebraic intersection with $ S $.  Therefore $S $ is separating.\end{proof}

For each such $S $, let $P_i^ S $ be the closure of the component of $M_i \setminus f_i (S) $ that does not contain $f_i (C_0) $.  Then if $E _{cc } $ is the union of the components of $M_A \setminus C $ that are neighborhoods of convex cocompact ends, one can construct a new $3$-manifold \vspace{1mm} $$M_A' = \big ( C \cup E_{cc }\big) \sqcup_{ f_i } \big (\bigcup_S P_i^S \big ) $$ by gluing each $P_i^ S $ to $C \cup E_{cc } $ along $S $.  The map $f_i $ extends naturally to a continuous map $\bar f_i: M_A' \to M_i; $ the extension into $E_{cc } $ is radial and on $P_i^S $ we use the natural inclusion into $M_i $.  It is easy to see that $\bar f_i $ is a $\pi_1 $-surjective covering map, so the proof ends the same way it did in the previous case and $M_A = M'_A $. \qed

\vspace{2mm}
We would like to observe that we did not really use that every maximal cyclic group in $\rho(\Gamma)$ is maximal cyclic in $G$. We namely proved the following less aesthetically pleasant but more general theorem:

\begin{sat}\label{max-cyclic-ugly}
Let $\Gamma$ be a finitely generated group and $(\rho_i)$ a sequence in $\calD (\Gamma) $. Assume that $ (\rho_i) $ converges algebraically to a representation $\rho$ and geometrically to a subgroup $G$ of $\PSL_2\BC$. If 
\begin{itemize}
\item $\rho(\Gamma)$ does not contain parabolic elements, and 
\item every degenerate end of $\BH^3/\rho(\Gamma)$ has a neighborhood which embeds under the covering $\BH^3/\rho(\Gamma)\to\BH^3/G$,
\end{itemize}
then $G=\rho(\Gamma)$. \qed
\end{sat}

Before concluding this section, observe that Theorem \ref{max-cyclic} together with Theorem \ref{anderson} imply the Anderson-Canary Theorem \ref{AC} mentioned in the introduction.

\section{Attaching roots}\label{sec:attaching-roots}
In this section we prove:

\begin{prop}\label{attaching-roots}
Let $M$ and $N$ be hyperbolic 3-manifolds with infinite volume and let $\tau:M\to N$ be a covering.  Assume that $N$ has no cusps, that $\pi_1(M)$ is finitely generated and that $M$ has a degenerate end which does not embed under the covering $\tau$. Then there is a hyperbolic 3-manifold $M'$ with finitely generated fundamental group, with $$|\chi(M')|<|\chi(M)|$$ and coverings $\tau':M\to M'$ and $\tau'':M'\to N$ with $\tau=\tau''\circ\tau'$.
\end{prop}

Note that since $M $ and $M' $ are not closed, the ratio of $\chi (M)  $ and $\chi(M') $ is not necessarily the degree of the covering $M \to  M'  $.

\vspace{2mm}

Continuing with the notation above, by Proposition \ref{exhaustion} the manifold $M$ has a standard compact core $C$ with the property that if a component $S$ of $\D C$ faces a degenerate end of $M$ then the restriction of $\tau$ to $S$ is a covering onto an embedded surface in $N$. The assumption that $M$ has a degenerate end which does not embed under the covering $\tau$ implies that there is actually a component $S_0$ of $\D C$ such that 
$$\tau\vert_{S_0}:S_0\to\tau(S_0)$$
is a non-trivial covering. Observe that by the covering theorem the embedded surface $\tau(S_0)\subset N$ faces a degenerate end of $N$.

Choosing a base point $*\in S_0$, we set $$\Gamma =\pi_1 (M,*)\ \text { and } \  H = \pi_1 (\tau (S_0),\tau (*)).  $$  The desired manifold $M' $ will be the cover of $ N $ corresponding to the subgroup $$\Gamma' = \left < \tau_*(\Gamma), H \right >\subset \pi_1(N,\tau(*)).$$  By construction, $ \pi_1(M') \cong \Gamma'$ is finitely generated and there are covering maps $\tau':M\to M'$ and $\tau'':M'\to N$ with $\tau=\tau''\circ\tau'$, so it remains only to prove that $|\chi(M')|<|\chi(M)| $.  

% under these maps, the surface $ S_0 \subset M $ nontrivially covers an embedded surface $S_0'  \subset M' $, which then embeds into $N $.  Note that the image of $\pi_1 (S_0',\tau' (*)) $ in $\Gamma' = \pi_1 (M',\tau' (*)) $ is the subgroup $H $ mentioned above.

For our purposes, the most useful way to interpret the Euler characteristic will be through its relation to the dimension of the deformation spaces $\calcium (M) $ and $\calcium (M') $ of convex-cocompact hyperbolic structures on $M$ and $M' $ (see Section \ref{conformalboundaries}).  Observe that since $M$ and $M' $ have finitely generated fundamental group and no cusps, they are homeomorphic by the tameness theorem to the interiors of compact hyperbolizable $3$-manifolds $\bar M $ and $\bar M' $ whose boundary components have negative Euler characteristic. It follows from Section \ref{conformalboundaries} that $\calcium (M) $ and $\calcium (M') $ are complex manifolds of $\BC $-dimensions $ 3|\chi (M) |$ and $3|\chi (M)| $, respectively, and that there is a holomorphic map $$ (\tau') ^* : \calcium (M') \to \calcium (M) $$ defined by lifting hyperbolic structures using $\tau' : M \to M' $.  We will prove:

\begin{claim} $ (\tau') ^* $ has discrete fibers and is not open.
\end{claim}
Since any holomorphic map with discrete fibers is open unless the dimension of the domain is smaller than the dimension of the image, we deduce from the claim that
$$3 |\chi(M')|=\dim_{\BC }\calcium (M') < \dim_{\BC } \calcium (M) =3|\chi(M)|.  $$
Therefore, $|\chi (M')| <| \chi (M) |$ and hence Proposition \ref{attaching-roots} will follow once we prove the claim.
\medskip

The first part of the claim is almost immediate.  If $\tau'_*: \Gamma \to\Gamma' $ is the inclusion induced by the covering $\tau' : M \to M' $ and $H < \Gamma' $ is as above, then by construction

\begin{itemize}
\item $H $ is torsion-free,
\item $\Gamma'$ is generated by $\tau'_*(\Gamma)$ and $H $,
\item $\tau'_* (\Gamma)\cap H $ has finite index in $H $.
\end{itemize} 

It follows from Corollary \ref{extension-finite} that a faithful representation $ \Gamma \to \PSL_2\BC $ has only finitely many extensions to $\Gamma' $.  Therefore, hyperbolic structures on $M' $ that map under $ (\tau')^*$ to the same element of $\calcium (M) $ have only finitely many options for holonomy representations, up to conjugacy.  However, the elements of $\calcium (M') $ with holonomy in any fixed conjugacy class form a discrete subset of $\calcium (M') $ (see page 154, \cite{Japaner}), so $ (\tau')^*$ must have discrete fibers.

To show that it is not open, we use the Ahlfors-Bers parameterization to produce from $(\tau')^*$ a holomorphic map 
\begin{displaymath}
\xymatrix{
\beta : \CT (\partial \bar M') \cong \calcium (M') \ar[r]^{  \ (\tau')^*}  & \calcium (M) \cong \CT (\partial \bar M).  }
\end{displaymath}
The Teichmuller spaces of $\partial \bar M' $ and $\partial \bar M $ split as products of the Teichmuller spaces of their connected components; let $S \subset \partial \bar M $ be the component adjacent to the degenerate end that the surface $S_0 $ faces.  The covering theorem implies that $\tau' $ extends to a nontrivial cover $\bar \tau: S \to S' $ onto some connected component $S' \subset \partial \bar M' $.  With respect to the decompositions 
$$\CT(\D\bar M)=\CT(S)\times\CT(\D\bar M\setminus S),\ \ \CT(\D\bar M')=\CT(S')\times\CT(\D\bar M'\setminus S')$$
the map $\beta $ can be written as 
$$\beta (\sigma_1,\sigma_2)=(\bar\tau^*\sigma_1,\hat \beta (\sigma_1,\sigma_2))$$
where $\bar\tau^*:\CT(S')\to\CT(S)$ is the map induced by the covering $\bar\tau: S\to S' $.  This covering is non-trivial, so the image of $\bar\tau^*$ has positive codimension and hence the same holds for the image of $\beta $.  Therefore $\beta $ is not open, implying the same for $(\tau')^*$. \qed
\medskip

\begin{bem}
A purely homological computation yields that under the same assumptions as in Proposition \ref{attaching-roots} we have $b_1(M')<b_1(M)$ where $b_1(\cdot)$ is the first Betti number with $\BR$-coefficients. However, this homological argument does not seem to work in the relative case that we will discuss in section \ref{sec:cusps}. This is why we choose to work with Euler characteristics and deformation spaces instead.
\end{bem}

\section{Proof of Theorem \ref{main}}\label{sec:main}
Recall the statement of Theorem \ref{main}:

\begin{named}{Theorem \ref{main}}
Let $\Gamma$ be a finitely generated group and $(\rho_i)$ a sequence in $\calD(\Gamma)$.  Assume that $ (\rho_i) $ is algebraically convergent and converges geometrically to a subgroup $G$ of $\PSL_2\BC$. If $G$ does not contain parabolic elements, then $G$ is finitely generated.
\end{named}

If the hyperbolic 3-manifold $\BH^3/G$ is compact then $G$ is obviously finitely generated. Assume from now on that this is not the case; since $G$ does not contain parabolic elements, this assumption implies that $\BH^3/G$ has infinite volume.

Among all finitely generated subgroups of $G$ which contain $\rho(\Gamma)$, choose $H$ such that the Euler characteristic of the associated hyperbolic 3-manifold $\BH^3/H$ has minimal absolute value. Since $H\subset G$ we have a covering $\BH^3/H\to\BH^3/G$. By minimality, we obtain from Proposition \ref{attaching-roots} that every degenerate end of $\BH^3/H$ embeds under this cover. 

On the other hand, the assumption that $H$ contains $\rho(\Gamma)$ implies, by Proposition \ref{competitors}, that there is a sequence of representations $\sigma_i:H\to\PSL_2\BC$ converging algebraically to the inclusion of $H$ in $G$ such that the groups $\sigma_i(H)$ converge geometrically to $G$. Theorem \ref{max-cyclic-ugly} implies now that $H=G$. In particular, $G$ is finitely generated. This concludes the proof of Theorem \ref{main}.\qed

\section{Parabolics}\label{sec:cusps}
It is a well established fact that most theorems in the deformation theory of Kleinian groups that hold in the absence of parabolics hold also, in some probably weaker form, in the presence of parabolics. It is also well known that proofs in the case with parabolics are much more cumbersome and technical but follow the same arguments as the proofs in the purely hyperbolic case. For the sake of clarity and transparency of exposition, we decided to prove Theorem \ref{main} only in the absence of parabolics. We state now the general results and discuss what changes have to be made in the arguments given above.

Throughout this section we assume that the reader is familiar with basic geometric facts about hyperbolic manifolds with cusps, \cite{Japaner}.
\medskip

As mentioned in the introduction, Evans \cite{Evans} obtained the following extension of Theorem \ref{AC}. 

\begin{sat}[Evans \cite{Evans}]\label{Evans}
Assume $\Gamma $ is a finitely generated group and that $(\rho_i)$ is a sequence of faithful representations in $\calD (\Gamma) $ converging algebraically to some representation $\rho$.  If the convergence $\rho_i\to\rho$ is weakly type preserving, then $\rho_i $ converges geometrically to $\rho(\Gamma)$.
\end{sat}

An algebraically convergent sequence $\rho_i\to\rho$ is {\em weakly type preserving} if for every $\gamma\in\Gamma$ such that $\rho(\gamma)$ is parabolic, there is some $i_\gamma \in \BN$ such that $\rho_i(\gamma)$ is parabolic for all $i\ge i_\gamma$. 

Theorem \ref{max-cyclic} can be extended to this setting as follows:

\begin{sat}\label{max-cyclic-cusp}
Let $\Gamma$ be a finitely generated group, and $(\rho_i)$ a sequence in $\calD (\Gamma) $ converging algebraically to a representation $\rho$ and geometrically to a subgroup $G$ of $\PSL_2\BC$. If 
\begin{itemize}
\item the convergence $\rho_i\to\rho$ is weakly type preserving, and
\item maximal cyclic subgroups of $\rho(\Gamma)$ are maximally cyclic in $G$,
\end{itemize}
then $G=\rho(\Gamma)$.
\end{sat}

As before, we will actually prove a stronger statement.

\begin{sat}\label{max-cyclic-cusp2}
Let $\Gamma$ be a finitely generated group, and $(\rho_i)$ a sequence in $\calD (\Gamma) $ converging algebraically to a representation $\rho$ and geometrically to a subgroup $G$ of $\PSL_2\BC$. If 
\begin{itemize}
\item the convergence $\rho_i\to\rho$ is weakly type preserving, and
\item every degenerate NP-end of $\BH^ 3 / \rho (\Gamma) $ has a neighborhood that embeds under the covering $\BH^ 3 / \rho (\Gamma)  \to \BH^ 3 / G $.
\end{itemize}
then $G=\rho(\Gamma)$.
\end{sat}

Recall that an NP-end of a hyperbolic $3$-manifold $M $ is an end of the manifold obtained  from $M $ by removing neighborhoods of all of its cusps.  The full statement of Canary's Covering Theorem \cite {Canary-covering} shows that if a covering $M \to N $ of infinite volume hyperbolic $3$-manifolds is not one-to-one on some degenerate NP-end, then some maximal cyclic subgroup of $\pi_1 M $ is not maximal cyclic in $\pi_1 N $.   One can see \cite {Japaner} for more details on the covering theorem and NP-ends.

\begin{proof} Let $M_i=\BH^3/\rho_i(\Gamma)$, $M_A=\BH^3/\rho(\Gamma)$, $M_G=\BH^3/G $ and let $\pi: M_A \to M_G $ be the covering induced by the inclusion $\rho (\Gamma) \subset G $.  We first recall the basic idea of the proof in the case without cusps.  First, there are arbitrarily large compact cores $C$ of $M_A$ such that the components of $\D C$ facing convex cocompact ends of $M_A$ are strictly convex. Using our second hypothesis above, we can also arrange that the components of these cores facing degenerate ends embed under the covering $\pi : M_A \to M_G$. Composing the restriction $\pi |_C $ with the almost isometric embeddings $\pi (C) \to M_i $ supplied by geometric convergence, we obtain maps $ f _i :C \to M_i$ such that 
\begin{enumerate}
\item if a component $S$ of $\D C$ faces a degenerate end of $M_A$ then $ f_i\vert_S$ is an embedding for all large $i$,
\item if $S \subset \partial C$ faces a convex cocompact end then $ f_i\vert_S$ is a convex immersion.
\end{enumerate}
We then construct for large $i$ a 3-manifold $N_i$ containing $C$ and a covering $\bar f_i:N_i\to M_i$ with $\bar f_i\vert_C= f _i $. This covering is $\pi_1$-surjective and hence a diffeomorphism, so in particular $C$ embeds under the covering $\pi:M_A\to M_G$. Since $C$ can be chosen to be arbitrarily large, this proves that the covering $\pi$ is trivial and hence $G=\rho(\Gamma)$.

If there are parabolics in the algebraic limit, the ends of $M_A $ are more complicated and refinements of the tools above are needed. The natural replacement of the compact core $C$ is a submanifold $C\subset M_A$ with the following properties:
\begin{enumerate}
\item if a component $S$ of $\D C$ faces a degenerate NP-end of $M_A$ then $S$ embeds under the covering $M_A\to M_G$,
\item if $S$ faces a geometrically finite NP-end then $S$ is strictly convex,
\item the complement in $C$ of the $\mu$-cuspidal part $M_A^{\cusp<\mu}$ of $M_A$ is a standard compact core, where $\mu$ is the Margulis constant.
\end{enumerate}

The construction of such a submanifold is the same as that used to produce the compact cores above: one takes a large metric neighborhood of the convex core $CC (M_A) $ and deletes standard neighborhoods of the degenerate NP-ends of $M_A $.  In this case, however, the resulting manifold $C $ will contain parts of the cusps of $M_A $, and will therefore be noncompact.  This is a problem, since $\pi (C) \subset M_G$ will also be noncompact and the almost isometric embeddings from $M_G $ to $M_i $ provided by geometric convergence are only defined on compact sets.  However, we still can use the almost isometric embeddings to produce locally bilipschitz maps $ f_i: C \setminus M_A^ {\cusp <\mu } \to M_i, $ and we will show that if the convergence $\rho_i \to\rho $ is weakly type preserving then these can be extended to locally bilipschitz maps $$\hat f_i:C\to M_i $$
converging uniformly on compact sets to the restriction $\pi |_ C$. The proof of Theorem \ref{max-cyclic-cusp} is then word-by-word the same as the proof of Theorem \ref{max-cyclic} with the maps $\hat f_i$ playing the role of $ f_i$. 

It remains to construct $\hat f_i$.  Consider the maps $$ f_i: C \setminus M_A^ {\cusp <\mu}\to M_i $$ described above, and let $\epsilon $ be a small positive constant.  If $i $ is large then $f_i $ is locally $( 1+\epsilon) $-bilipschitz; combined with the fact that the convergence $\rho_i \to \rho $ is weakly type preserving and $M_A $ has finitely many cusps, this implies that $f_i $ sends loops homotopic into the cusps of $M_A $ to parabolic loops in $M_i $ with nearly the same length.  It follows that $C\cap \partial M_A^{\cusp<\mu}$ is mapped under $ f_i $ into a small neighborhood of $\partial M_i^{\cusp<\mu}$.  After a small perturbation, we can then arrange that

\begin{enumerate}
\item $f_i (C \cap \partial M_A^ {\cusp <\mu }) \subset \partial M_i^ {\cusp <\mu } $,
\item $D f_i $ sends vectors orthogonal to $\partial M_A^ {\cusp < \mu } $ to vectors orthogonal to $\partial M_i^ {\cusp <\mu } $.
 \end{enumerate} 
Brock and Bromberg \cite[Lemma 6.16]{Brock-Bromberg} show how to accomplish such a perturbation with a bit of finesse; in particular, their argument shows that we can still assume that $f_i $ is locally bilipschitz.

Recall that $C $ was constructed by removing standard neighborhoods of the degenerate NP-ends of $M_A $ from its convex core.  These neighborhoods can be chosen so that their intersections with $M_A^ {\cusp <\mu } $ are foliated by geodesic rays orthogonal to $\partial M_A^ {\cusp <\mu } $; the intersection $C \cap M_A^ {\cusp <\mu} $ then enjoys the same property.  Combining this with (1) and (2) above allows us to extend $f_i $ to a locally bilipschitz map $\hat f_i : C \to M_i $ as follows: define $\hat f_i $ to coincide with $f_i $ on $C\setminus M_A^{\cusp<\mu}$ and map geodesic rays in $C \cap M_A^ {\cusp <\mu } $ orthogonal to $ \partial M_A^ {\cusp <\mu } $ isometrically to geodesic rays orthogonal to $\partial M_i^{\cusp<\mu}$.  A quick computation in the upper half space model for $\Hyp^3 $ verifies that the extension $\hat f_i $ is also locally $(1+\epsilon) $-bilipschitz.  It follows that $\hat f_i \to \pi |_C $ uniformly on compact subsets.d
\end {proof}

Observe that as in the case without parabolics, we can replace the second hypothesis of Theorem \ref{max-cyclic-cusp} with the condition that every degenerate NP-end of $M_A $ has a neighborhood that embeds in $M_G $.
\medskip

A version of Proposition \ref{attaching-roots} that applies to manifolds with cusps is also readily established:

\begin{prop}\label{attaching-roots-cusp}
Let $M$ and $N$ be hyperbolic 3-manifolds with infinite volume and let $\tau:M\to N$ be a covering.  Assume that $\pi_1(M)$ is finitely generated and that $M$ has a degenerate NP-end that does not embed under the covering $\tau$. Then there is a hyperbolic 3-manifold $M'$ with finitely generated fundamental group, with 
$$3|\chi(M')|-\#\{\hbox{cusps in}\ M'\}<3|\chi(M)| -\#\{\hbox{cusps in}\ M\}$$ 
and coverings $\tau':M\to M'$ and $\tau'':M'\to N$ with $\tau=\tau''\circ\tau'$.
\end{prop}

The proof is the same as that of Proposition \ref{attaching-roots}, except that instead of considering the deformation space $\calcium (M) $ of convex-cocompact hyperbolic structures on $M$ one uses geometrically finite metrics whose parabolic loci coincide with that of the original hyperbolic structure on $M$.  The space of such metrics, up to isometries isotopic to the identity map, is by \cite [Theorem 8.44] {Kapovich} a complex manifold of $\BC $-dimension $$-3 \chi(M)-\#\{\hbox{cusps in}\ M\}\ge 0.  $$
After these remarks the proof of Proposition \ref{attaching-roots-cusp} is the same as the proof of Proposition \ref{attaching-roots}.\qed
\medskip

Having provided versions of Theorem \ref{max-cyclic} and Proposition \ref{attaching-roots} that apply to representations with parabolics, we are almost ready to discuss the general form of Theorem \ref{main}. Before doing so, we need a definition:

\begin{defi*}
Assume that a sequence of subgroups $(G_i)$ of $\PSL_2\BC$ converges geometrically to a subgroup $G$. We say that the convergence $G_i\to G$ is {\em geometrically weakly type preserving} if for every $g\in G$ parabolic there is a sequence $(g_i)$ with $g_i\in G_i$ converging to $g$ and with $g_i$ parabolic for all but finitely many $i$.
\end{defi*}

In the terminology of \cite{BS}, a geometrically convergent sequence of subgroups converges in a geometrically weakly type preserving manner if and only if the associated sequence of pointed manifolds has {\em uniform length decay}. 

The general version of Theorem \ref{main} reads now:

\begin{sat}\label{main-cusp}
Let $\Gamma$ be a finitely generated group and $(\rho_i)$ a sequence in $\calD (\Gamma) $ that converges algebraically to a representation $\rho$ and geometrically to a subgroup $G$ of $\PSL_2\BC$. If the convergence $\rho_i\to G$ is geometrically weakly type-preserving, then $G$ is finitely generated.
\end{sat}
\begin {proof}
As in the proof of Theorem \ref{main}, we choose a finitely generated subgroup $H\subset G$ containing $\rho(\Gamma)$ and minimizing the quantity $$3|\chi(\BH^3/H)| -2\#\{\hbox{cusps in}\ \BH^3/H\} \geq  0. $$ 

We claim that $G=H$. By Proposition \ref {competitors}, there are for large $i $ representations $\sigma_i:H\to\PSL_2\BC$  with $\sigma_i(H)=\rho_i(\Gamma_i)$ converging algebraically to the inclusion of $H$ in $G$. In particular, $G$ is the geometric limit of the groups $\sigma_i(H)$. The assumption that the convergence $\sigma_i(H)=\rho_i(\Gamma)\to G$ is geometrically weakly type-preserving implies that the algebraic convergence of $\sigma_i$ to the inclusion of $H$ into $G$ is (algebraically) weakly type preserving. In particular we deduce from Theorem \ref{max-cyclic-cusp2} that either $G=H$ or some degenerate NP-end of $\BH^3/H$ has no neighborhood that embeds under the cover $\BH^3/H\to\BH^3/G$.  The second possibility is ruled out by the choice of $H$ and Proposition \ref{attaching-roots-cusp}, so $G =H $ and therefore is finitely generated. 
\end{proof}

\appendix

\section{Some assorted results}

As mentioned in the introduction, we discuss here to which extent some other results concerning faithful representations remain true if the condition of faithfulness is relaxed. 

\begin{defi*}
A sequence $(\rho_i)$ of representations is {\em eventually faithful} if for all $\gamma\in\Gamma \setminus\{1\}$ there is some $i_\gamma$ such that $\gamma\notin\Ker(\rho_i)$ for all $i\ge i_\gamma$.
\end{defi*}

Many algebraically convergent sequences determine an eventally faithful sequence.  The following well-known lemma formalizes this; its proof is a simple application of 
 the J\o rgensen inequality.

\begin{lem} Let $(\rho_i) $ be a sequence of representations in  $\mathcal D (\Gamma) $ that converges algebraically to a representation $ \rho$. Then given any   $\gamma  \in \Gamma  $, we have that $\gamma \in \ker \rho \Longleftrightarrow \exists i_0\in  \BN  \text { such that  } \gamma \in \ker \rho_i \text { for all } i \geq  i_0.$ 
\label{kernelslemma}
\end{lem}

Therefore, if $\ker \rho $ is finitely normally generated  then for large $i $ the representations  $\rho_i $ factor through $\Gamma/ \ker (\rho)  $; the resulting sequence of representations of $\Gamma/ \ker (\rho)  $ will be eventually faithful.

\begin{proof}
One direction of the implication is clear.  For the other direction, we assume that $\gamma \in  \ker  \rho $ and try to prove that $\gamma \in \ker \rho_i $ for large $i $.  

As $\mathcal D (\Gamma) $ is closed the representation $\rho $ is non-elementary, so there are two elements $\alpha_1,\alpha_2 \in \Gamma $ such that $ \rho (\alpha_1) $ and $ \rho (\alpha_2) $ are isometries of $\Hyp^3 $ of hyperbolic type and have distinct axis. Since $$\rho_i (\gamma)  \longrightarrow \rho (\gamma) = \Id, $$ we have for sufficiently large $i $ that each of the two pairs $\{ \rho_i (\gamma), \rho_i(\alpha_k)\} $ violates the J\o rgensen inequality (Theorem 2.17 in \cite{Japaner}). Therefore, both groups $\left< \rho_i (\gamma), \rho_i (\alpha_k) \right>$, $k =1,2 $, are abelian.  But since $ \rho_i (\alpha_1) $ and $ \rho_i (\alpha_2) $ have different axes the only way that both of these groups can be abelian is if $ \rho_i (\gamma) $ is elliptic or trivial. This is a contradiction, since it is nontrivial by assumption and cannot be elliptic since $ \rho_i $ is discrete and torsion-free.
\end{proof}

We assume from now on that $(\rho_i)$ is an eventually faithful sequence of representations in $\calD (\Gamma) $.  If $ S \subset \Gamma $ is a finite generating set, then each representation $\rho_i $ determines a convex function $$d_{\rho_i }:\BH^3\to\BR, \, \, d_{\rho_i } (x)=\sum_{\gamma\in S}d_{\BH^3}(x,\rho_i (\gamma)x).  $$ 
Conjugating our representations by $\PSL_2\BC $ if necessary, we may assume that some base point $0\in\BH^3$ is the unique minimum of each $d_{\rho_i } $.  In this case we say that the sequence $(\rho_i)$ consists of {\em normalized representations} and we set $d_{\rho_i}=d_{\rho_i}(0)$. It is well-known that the sequence $(\rho_i)$ contains an algebraically convergent subsequence if 
$$\liminf d_{\rho_i}<\infty.  $$
Otherwise, the sequence of actions of $\Gamma$ via $\rho_i$ on the scaled hyperbolic spaces $\frac 1{d_{\rho_i }}\BH^3$ contains a subsequence which converges in the equivariant Gromov-Hausdorff topology to a non-trivial action $\Gamma\actson T$ on some real tree $T$. Recall that an action on a tree is called trivial if it has global fixed points. Morgan-Shalen \cite{Morgan-Shalen}, Paulin \cite{Paulin} and Bestvina \cite{Bestvina} proved that if the representations $\rho_i $ are faithful the action $\Gamma\actson T$ is {\em small}, meaning that the stabilizers of non-degenerate segments in $T$ are virtually abelian. Their arguments still apply if the sequence is only eventually faithful, see \cite{Hossein-Souto}:

\begin{sat}\label{tree}
Every eventually faithful sequence of normalized representations in $\calD (\Gamma) $ has a subsequence that either converges algebraically in $\calD (\Gamma) $ or converges in the Gromov Hausdorff topology to a nontrivial small action $\Gamma\actson T$ on a $\BR$-tree $T$.\qed
\end{sat}

It is a theorem of Morgan-Shalen \cite{Morgan-Shalen} that if the fundamental group of a compact, irreducible and atoroidal $3$-manifold $M$ admits a nontrivial small action on a real tree, then either $\D M$ is compressible or there are essential properly embedded annuli in $(M,\D M)$. Combining this fact with Theorem \ref{tree} we obtain the following result, essentially due to Thurston:

\begin{kor}\label{thurston}
Assume that $\Gamma$ is the fundamental group of a compact 3-manifold with incompressible and acylindrical boundary. Then every eventually faithful sequence of normalized representations in $\calD (\Gamma) $ contains a convergent subsequence.\qed
\end{kor}

Many other results, for instance Thurston's double limit theorem \cite{Otal96} or the results in \cite{KS02}, ensuring the existence of convergent subsequences of sequences of faithful representations can be reduced to the non-existence of certain actions of groups on trees; variants of all these results still hold for eventually faithful sequences. It may be therefore surprising that some convergence results for sequences of faithful representations completely fail in our more general setting.

In \cite{Thu86A}, Thurston proved a generalization of Corollary \ref{thurston} to the case that $\Gamma$ is the fundamental group of a compact 3-manifold $M$ with incompressible boundary. More precisely, the so-called {\em only-windows-break theorem} asserts that whenever $(\rho_i)$ is a sequence of faithful representations of $\Gamma=\pi_1(M)$ and $N\subset M$ is a component of the complement of the characteristic manifold, then the sequence $(\rho_i\vert_{\pi_1(N)})$ has, up to conjugacy, a convergent subsequence.  Leaving the interested reader to consult \cite{Jaco} for information about the characteristic manifold, we limit ourselves to the following concrete example.

Let $H$ be a handlebody of genus $2$ and $\gamma\subset\D H$ a simple closed curve on the boundary of $H$ such that $\D H\setminus\gamma$ is incompressible and acylindrical; for instance, $\gamma$ can be taken in the Masur domain of $H$ \cite{KS02}. We consider the manifold $ N $ obtained by doubling $H$ along $\CN(\gamma)$, where $\CN(\gamma)$ is a regular neighborhood of $\gamma$ in $\D H$.   The choice of $\gamma $ ensures that $N$ has incompressible boundary and that there is a unique  (up to isotopy) properly embedded essential annulus $A\subset N$, the annulus along which we have glued. The annulus $A$ cuts $N$ open into two copies $H_1$ and $H_2$ of $H$. 

In this example, Thurston's only-windows-break theorem asserts:

\begin{sat}[Thurston]\label{owb}
Let $N$ be as above and $ (\rho_i) $ a sequence of discrete and faithful representations of $\pi_1(N) $ into $\PSL_2\BC $.  Then the sequence of restrictions $$\rho_i |_{\pi_1 (H_1) } : \pi_1 (H_1 )\to \PSL_2\BC $$ \vspace{2mm} has a subsequence that converges up to conjugacy in $\PSL_2\BC $. \end{sat}

We claim that there is an eventually faithful sequence of discrete representations of $\pi_1(N)$ for which the claim of Theorem \ref{owb} fails. 

By construction, the manifold $N$ is the double of $H$ along $\CN(\gamma)$. Let 
$$\tau:N\to H$$
be the map given by `folding' along $\CN(\gamma)$. Identifying $H$ with $H_1$, one of the two pieces of $N$, we choose a base point $p\in \CN  (\gamma)$ and hence we have a homomorphism
$$\tau_*:\pi_1(N,p)\to\pi_1(H,p)$$
We consider also the Dehn-twist 
$$\delta:N\to N$$
along the annulus $A$. 

\begin {claim}
The sequence 
$\tau_*\circ\delta_*^n:\pi_1(N,p)\to\pi_1(H,p)$ is eventually faithful.
\end{claim}

It is worth observing that this claim is really a statement about a `twist' automorphism of any amalgamation of a free group with itself over some nontrivial cyclic subgroup.

\begin {proof}
The fundamental group $\pi_1  (N, p) $ is isomorphic to the amalgamation  $\pi_1 (H_1, p) \star_{ \left  < \gamma \right >} \pi_1 (H_2, p),  $ where with an abuse of notation we regard $\gamma $ as an element of both $\pi_1 (H_1, p) $ and $\pi_1 (H_2, p) $.  The folding map $\tau_*  : \pi_1 (N, p) \to \pi_1 (H, p)$ is the unique map from the amalgamation to $\pi_1 (H, p) $ that restricts to the canonical identifications $\pi_1 (H_i, p) \to \pi_1 (H, p) $.

Fix a free basis for $\pi_1 (H) $. There are then associated free bases for $\pi_1 (H_1, p) $ and $\pi_1 (H_2, p) $, and after a conjugation any element of $\pi_1 (N, p) $ can be written as
$$w = a_1  b_1 a_2 b_2\ldots a_k b_k,$$
where $a_1,\ldots, a_k $ and $b_1,\ldots, b_k $ are reduced words in the generators of $\pi_1 (H_1, p) $ and $\pi_1 (H_2, p)$, respectively, and all are nontrivial except possibly $a_1$.   Regarding $\gamma $ as a cyclically reduced word in the generators of $\pi_1 (H, p) $, we can ensure by choosing $k $ to be minimal that none of the words $a_2,\ldots, a_k ,b_1,\ldots, b_k $ correspond to elements of $\pi_1 (H, p) $  that  are powers of $\gamma $.  Then the $n $-fold twist of $w $ has the following expression:
$$\delta_*^  n (w) = a_1 (\gamma^  n b_1 \gamma^  {- n}) a_2 (\gamma^  n b_2  \gamma^  {- n} )\ldots a_n ( \gamma^  n b_n \gamma^  {- n}). $$

After applying the folding map $\tau_*$, this becomes
$$\tau_*  \circ\delta_*^  n (w) = a_1'  (\gamma^  n b_1'  \gamma^  {- n}) a_2'  (\gamma^  n b_2'   \gamma^  {- n} )\ldots a_n'  ( \gamma^  n b_n'  \gamma^  {- n}), $$ where $a_1' ,b_1' ,\ldots a_k' , b_k'  $ are the images of $a_1 ,b_1,\ldots a_k , b_k$ under the identifications $\pi_1(H_i, p) \to \pi_1 (H, p) $.  However, as  none of $b_1',a_2' ,b_2' ,\ldots a_k' , b_k'  $ are powers of the cyclically reduced word $\gamma $, there is only a bounded  amount of cancellation when reducing the word above.  Therefore, if $n $ is large then $\tau_*  \circ\delta_*^  n (w)$ is nontrivial.
\end{proof}

\begin{bem}
Essentially we are saying that $\pi_1(N, p)$ is fully residually free. This fact holds for every group obtained by doubling a free group along a cyclic subgroup.
\end{bem}

Since $H$ is a handlebody, $\pi_1(H,p)$ is a free group. Hence, we can choose some sequence of faithful representations $(\sigma_n)\in\calD(\pi_1(H, p))$ which cannot be conjugated to obtain convergent subsequences.  By construction, $\tau_*\circ\delta_*^n$ is the identity on $\pi_1(H,p)$ for all $n$. So setting $$\rho_n=\sigma_n\circ\tau_*\circ\delta_*^n : \pi_1 (N, p) \longrightarrow \PSL (2,\BC),$$ we obtain an eventually faithful sequence of representations which does not contain any subsequence whose restriction to $\pi_1(H, p)$ converges up to conjugacy.

\begin{sat}\label{no-owb}
The only-windows-break theorem fails for eventually faithful sequences of representations.\qed
\end{sat}

An alternative statement of Theorem \ref{no-owb} could be that when one is not absolutely faithful more than the windows can get broken.

\bigskip

{\small 
\noindent Ian Biringer, Department of Mathematics, University of Chicago

\noindent \texttt{biringer@math.uchicago.edu}

\bigskip

\noindent Juan Souto, Department of Mathematics, University of Michigan, 

\noindent
\texttt{jsouto@umich.edu}}

\end{document}